\documentclass[12pt,centertags,reqno]{amsart}

\usepackage[foot]{amsaddr}
\usepackage{latexsym}
\usepackage[english]{babel}
\usepackage[T1]{fontenc}

\usepackage[utf8]{inputenc}
\usepackage[round]{natbib}
\usepackage{amssymb,amsmath}
\usepackage{fancyhdr}
\usepackage{url}
\usepackage{hyperref}
\usepackage{verbatim}
\usepackage{leftidx}
\usepackage{bm}
\usepackage{color,graphicx}
\usepackage{tikz}
\usepackage{bbm}

\usepackage{mathrsfs, mathtools}
\usepackage{stmaryrd}

\usepackage{marvosym}
\mathtoolsset{showonlyrefs}

\usepackage{appendix}

\usepackage{soul}

\usepackage{enumerate} 
\usepackage{enumitem} 

\numberwithin{equation}{section} \makeatletter
\renewcommand{\subsection}{\@startsection
{subsection}{2}{0mm}{\baselineskip}{-0.25cm}
{\normalfont\normalsize\bf}} \makeatother

\newtheorem{theorem}{Theorem}[section]
\newtheorem{lemma}[theorem]{Lemma}
\newtheorem{corollary}[theorem]{Corollary}
\newtheorem{definition}[theorem]{Definition}
\newtheorem{remark}[theorem]{Remark}
\newtheorem{proposition}[theorem]{Proposition}
\newtheorem{example}[theorem]{Example}
\newtheorem{assumption}[theorem]{Assumption}

\def \F {\mathcal F}
\def \G {\mathcal G}

\def \Q {\mathbf Q}

\def \R {\mathbb R}
\def \bF {\mathbb F}

\def \bE {\mathbb E}
\def \bN {\mathbb N}

\newcommand{\ud}{\mathrm d}
\newcommand{\ds}{\displaystyle}
\newcommand{\esp}[2][\mathbb E] {#1\left[#2\right]}

\newcommand{\condespf}[2][\F_t]       {\mathbb E\left.\left[#2\right|#1\right]}

\DeclareMathOperator*{\essinf}{ess\,inf}
\DeclareMathOperator*{\esssup}{ess\,sup}
\DeclarePairedDelimiter{\abs}{\lvert}{\rvert}

\def \a {{(1)}}
\def \b {{(2)}}
\def \i {{(i)}}

\newcommand{\Pp}{\mathbf P}

\newcommand{\Ll}{\mathcal L}

\def \a {{(1)}}
\def \b {{(2)}}
\def \i {{(i)}}

\begin{document}

\author[C.~Ceci]{Claudia  Ceci}
\address{Claudia  Ceci, Department MEMOTEF,
University of Rome Sapienza, Via del Castro Laurenziano, 9,
I-00161 Rome, Italy.}\email{claudia.ceci@uniroma1.it}
\author[A.~Cretarola]{Alessandra Cretarola\, 
}
\address{Alessandra Cretarola, Department of Economics, University ``G. D'Annunzio'' of Chieti-Pescara,
Viale Pindaro, 42, I-65127 Pescara, Italy.}\email{alessandra.cretarola@unich.it}

\title[Optimal reinsurance in a dynamic contagion model]{Optimal reinsurance in a dynamic contagion model: comparing self-exciting and externally-exciting risks}

\begin{abstract}
We investigate the optimal reinsurance problem in a risk model with jump clustering features.
This modeling framework is inspired by the concept initially proposed in  \citet{Dassios_Zhao_2011}, 
 combining Hawkes and Cox processes with shot noise intensity models. 
Specifically, these processes describe self-exciting and externally excited jumps in the claim arrival intensity, respectively.  
The insurer aims to maximize the expected exponential utility of terminal wealth 
for general reinsurance contracts and reinsurance premiums. We discuss two different methodologies: the classical stochastic control approach based on the Hamilton-Jacobi-Bellman (HJB) equation and a backward stochastic differential equation (BSDE) approach.
In a Markovian setting, differently from the classical HJB-approach, the BSDE method 
enables us to solve the problem without imposing any requirements for regularity on the associated value function.
We provide a Verification Theorem in terms of a suitable BSDE driven by a two-dimensional marked point process and we prove an existence result relaying on the theory developed in \citet{Papa_Possa_Sapla2018} for stochastic Lipschitz generators.  
After discussing the optimal strategy for general reinsurance contracts and reinsurance premiums, we provide more explicit results in some relevant cases. Finally, 
we provide comparison results 
that highlight the heightened risk stemming from the self-exciting component in contrast to the externally-excited counterpart and discuss the monotonicity property of the value function.  
\end{abstract}

\maketitle



\section{Introduction}

\noindent Optimal reinsurance and optimal investment problems for various risk models have gained a lot of interest in the actuarial literature in recent years. Thanks to the development of effective strategies, 
insurers can reduce potential claim risk (insurance risk) and optimize capital investments. Indeed, 
acquiring reinsurance serves as a safeguard for insurers against unfavorable claim experiences, while investing also enables insurers to diversify risks and potentially achieve higher returns on the cash flows within their insurance portfolio.
Within the extensive body of literature devoted to risk theory,
a classical task is to deal with optimal risk control and optimal asset allocation for an insurer. Mainly in the case of classical reinsurance contracts such as proportional and excess-of-loss, different decision criteria have been adopted in the study of these problems, e.g., ruin probability minimization, dividend optimization, and expected utility maximization.
Here, we focus on the latter approach (see \citet{Irgens_Paulsen,Mania2010,Brachetta_Ceci_2019} and references therein). Earlier seminal papers on the topic adopt a diffusive dynamics for the surplus process, whereas more recent literature explores surplus processes that incorporate jumps.\\
\indent The first risk model specification incorporating jumps in nonlife insurance is represented by the classical Cramér-Lundberg model, in which the claim arrival process follows a Poisson process with constant intensity.
Since it is an assumption which is seriously violated in a large number of insurance contexts (e.g., climate risks), many researchers have suggested to employ a stochastic intensity for the claim arrival dynamics. For instance, clustering features due to exogenous (externally excited) factors, such as earthquakes, flood, and hurricanes, might be captured using a Cox process; see, e.g. \citet{Albrecher_Asmussen,Bjork_Grandell,Embrechts_Schmidli_Grandell}. Moreover, clustering effects due to endogenous (self-excited) factors, such as aggressive driving habits and poor health conditions, can be effectively described by a Hawkes process, see e.g. \citet{Hawkes_1971}. A dynamic contagion model was introduced in 
\citet{Dassios_Zhao_2011} by generalizing both the Cox process with shot noise intensity and the Hawkes process.\\
In recent years, \citet{Cao_Landriault_Li} analyzed the optimal reinsurance-investment problem for the compound dynamic contagion process introduced by \citet{Dassios_Zhao_2011} via the time-consistent mean–variance criterion. Very recently, \citet{brachetta_call_ceci_sgarra}  investigated the optimal reinsurance strategy for a risk model with jump clustering characteristics similar to that proposed by \citet{Dassios_Zhao_2011} under partial information. \\
\indent In this work, 
we study the optimal reinsurance problem by maximizing the expected utility in the risk model with jump clustering properties introduced in \citet{brachetta_call_ceci_sgarra} with complete information for general reinsurance contracts. Note that, the  problem considered in \citet{brachetta_call_ceci_sgarra} is
the same but analyzed in a partial information setting. 
The study of the problem in the case of complete information is not addressed in the literature, and furthermore, it allows for comparative analyses between self-exciting and externally-exciting risks in a more tractable context than that of partial information. 
We discuss two different methodologies: the classical stochastic control approach based on the Hamilton-Jacobi-Bellman (HJB) equation and a backward stochastic differential equation  (BSDE) approach. 
It is important to stress that 
proving the existence of a classical solution to the HJB equation corresponding to the optimal stochastic control problem under investigation is challenging due to its inherent complexity. This difficulty stems from the equation's nature as a partial integro-differential equation, compounded by an optimization component embedded within the associated integro-differential operator.
Under specific assumptions and in a risk model where jump clustering is driven solely by a self-exciting component, \citet{wu2024optimal} derive the optimal reinsurance strategy and corresponding closed-form value function using the HJB approach. It is important to highlight that, unlike their model, our framework incorporates self-exciting jumps that depend on claim size, along with an additional externally-excited jump component. Moreover, while they focus exclusively on the expected value principle, we consider the problem in a more general setting without adopting a specific premium principle. Given the added complexity introduced by these features, the standard HJB framework becomes less tractable in our contagion model.
This motivated the application of 
 an alternative approach based on BSDEs. 
It should be noted that the resulting BSDE, whose unique solution characterizes the value process,
differs from that studied \citet{brachetta_call_ceci_sgarra}, due to the presence of an additional jump component. 
We provide a Verification result (see Theorem \ref{T1}) by proving that any solution to a suitable BSDE driven by two pure jump processes coincides with the Snell Envelope associated to null reinsurance.  
Next, we discuss existence of solution to our BSDE in Theorem \ref{T2}. Most of the literature on BSDEs with jumps requires the Lipschitz  property of the generator (see e.g. \citet{lim-quenez, jeanblanc2015, kazi-tani2015, cfj2016, abdel2022} and references therein).
It is important to note that due to the unboundedness of the claim arrival intensity, our BSDE satisfies only a stochastic Lipschitz condition. Therefore, we apply the theory developed in \citet{Papa_Possa_Sapla2018} for multidimensional BSDEs driven by a general martingale assuming a stochastic Lipschitz generator. The application of the aforementioned result also requires one to show the validity of the martingale representation property, see Proposition \ref{mg representation}.
 Theorems \ref{T1}, \ref{T2} are summarized in Corollary \ref{corollary}, where we characterize the value process and optimal strategies. In a Markovian setting we provide more insight into the structure of optimal strategies aligning with the results obtained via the HJB-approach under the additional regularity assumption of  the value function.  
 Our findings suggest that mitigating the risk stemming from externally-excited jumps can only be accomplished through adjusting the premium rate. While, mitigating the self-exciting effect requires adjustments to both the premium rate and the reinsurance strategy.
In addition, the case of Cox process with shot noise intensity has been discussed in detail and explicit expressions for the optimal strategy are provided in some cases of interest.
Finally, we explore the monotonicity property of the value function, which indicates a more conservative stance adopted by the insurer in the contagion model compared to the Cox model whether employing proportional reinsurance or limited excess of loss reinsurance with a fixed maximum coverage.\\
\indent The paper is organized as follows. Section \ref{sec:model} introduces the mathematical framework including the dynamic contagion process. Section \ref{sec:formulation} formally introduces the problem under investigation, which involves the controlled surplus process and the objective function. In Section \ref{HJB section} we discuss the HJB approach in order to solve the resulting optimal stochastic control problem. The characterization of the value process and the optimal strategy via a suitable BSDE can be found in Section \ref{sec:bsde}. Section \ref{Optimal_Reinsurance} provides the representation of the optimal reinsurance strategy in a Markovian framework for general premiums and more explicit results in some special cases. In Section \ref{sec:comparison} we perform a comparison analysis which confirms the risk due to the self-exciting component and discuss the monotonicity property of the value function.
Finally, all technical proofs and some auxiliary results are collected in Appendix \ref{appendix:proofs}.

\section{The mathematical framework}\label{sec:model}

Let $(\Omega,\F, \mathbf P;\bF)$ be a filtered probability space and assume that the filtration $\bF=\{\F_t, \ t \in [0,T]\}$ satisfies the usual conditions of completeness and right-continuity. Here, $T>0$ is a fixed time horizon that represents the maturity of a reinsurance contract. 


We consider the dynamic contagion process proposed in \citet{brachetta_call_ceci_sgarra},
which generalizes the Hawkes and Cox processes with shot noise intensity introduced by \citet{Dassios_Zhao_2011}. More precisely, the claim counting process $N^\a=\{N_t^\a,\ t \in [0,T]\}$ has the $(\bF,\Pp)$-stochastic intensity process $\Lambda=\{\lambda_t,\ t \in [0,T]\}$ given by
\begin{equation}\label{intensity}
\lambda_t = \beta + (\lambda_0 - \beta)  e^{-\alpha t} + \sum_{j=1}^{N^\a_t} e^{-\alpha (t -  T^\a_j)} \ell(Z^\a_j) +  \sum_{j=1}^{N^\b_t}  e^{-\alpha (t - T^\b_j)} Z^\b_j ,\quad t \in [0,T],
\end{equation}
where
\begin{itemize}
\item  $\beta > 0$ is the constant reversion level;
\item  $\lambda_0 >0$ is the initial value of $\Lambda$;
\item $\alpha >0$ is the constant rate of exponential decay;
\item $N^\b=\{N_t^\b,\ t \in [0,T]\}$ is a Poisson process with constant intensity $\rho>0$;
\item $\{T^\a_n\}_{n \geq 1}$ are the jump times of $N^\a$, i.e., the time instants when claims are reported;
\item $\{T^\b_n\}_{n \geq 1}$  are the jump times of $N^\b$, i.e., when exogenous/external factors make intensity jump;
\item $N^\a$ and $N^\b$ do not have common jump times;
\item $\{Z^\a_n\}_{n \geq 1}$ represent the claim size and they are modeled as a sequence of i.i.d. $\R^+$-valued random variables with  distribution function $F^\a : [0, + \infty) \to [0,1]$ such that $\mathbb E[Z^\a] < + \infty$;
\item $\ell : [0, +\infty) \to [0, +\infty)$ is a measurable function (for instance we could take $\ell(z) = a z$, $a>0$, and the self-exciting jumps would be proportional  to  claims sizes) such that $\mathbb E[ \ell(Z^\a)] < + \infty$;
\item $\{Z^\b_n\}_{n \geq 1}$ are the externally-excited jumps and they are modeled as a sequence of i.i.d. $\R^+$-valued random variables with  distribution function $F^\b : [0, + \infty) \to [0,1]$, such that $\mathbb E[Z^\b] < + \infty$.
\end{itemize}

\noindent Note that the counting process $N^{(1)}$ is defined via its intensity $\Lambda$ in equation \eqref{intensity}, which in turn depends on the history of $N^{(1)}$. So, an apparent logical loop seems to arise concerning the existence of $\Lambda$. For more details, refer to \citet{brachetta_call_ceci_sgarra}.
The following assumption will hold from now on:
\begin{assumption}\label{ass:indep}
We assume $N^{(2)}$,  $\{Z^\a_n\}_{n \geq 1}$ and $\{Z^\b_n\}_{n \geq 1}$ to be independent of each other.
\end{assumption}

\noindent We introduce the cumulative claim process $C = \{ C_t,\  t \in [0,T] \}$ defined  at time $t$ as
\begin{equation}\label{loss}
C_t = \sum_{j=1}^{N^\a_t} Z^\a_j,\quad t \in [0,T].
\end{equation}
and the integer-valued random measures $m^\i(\ud t, \ud z)$, $i=1,2$
\begin{equation}
\label{eqn:m1m2}
m^\i(\ud t, \ud z) = \sum_{n \geq 1} \delta_{(T_n^\i, Z_n^\i)}(\ud t, \ud  z) 1\!\!1_{\{ T_n^\i <+ \infty\}},
\end{equation}
where $\delta_{(t,z)}$ denotes the Dirac measure in $(t,z)$. We recall from \citet{brachetta_call_ceci_sgarra} the change of measure which allows to introduce the dynamic contagion model via a rigorous construction,  starting from two Poisson processes  $N^\a, N^\b$ with intensity $1$ and $\rho$ on a given probability
space $(\Omega,\F,\mathbf Q;\mathbb F)$ and two sequences of $\{Z^\a_n\}_{n \geq 1}$, $\{Z^\b_n\}_{n \geq 1}$ of i.i.d. positive random variables with distribution functions $F^\a$ and $F^\b$, respectively. We assume $N^\a$, $N^{(2)}$,  $\{Z^\a_n\}_{n \geq 1}$ and $\{Z^\b_n\}_{n \geq 1}$ to be independent of each other under $\mathbf Q$.  Specifically, 
under the following assumption:
\begin{assumption}\label{nuova_bis}
There exists $\varepsilon >0$ such that
$$\mathbb E^{\Q}\left[ e^{\varepsilon \ell(Z^\a)} \right] <+ \infty,
	\quad \mathbb E^{\Q} \left[ e^{\varepsilon Z^\b} \right] <+ \infty,$$
\end{assumption}
\noindent it is possible to define the equivalent probability measure $\Pp$ via
$
\frac{\ud \Pp}{\ud \Q} {\Big \vert  _{\mathcal F_T}} = L_T,
$
where $L_T$ is the final value of the $(\bF,\Q)$-martingale $L=\{L_t,\ t \in [0,T]\}$ given by
\begin{equation}
\label{eqn:L}
L_t = e^{ -\int_0^t (\lambda_{s} - 1) \ud s  + \int_0^t \ln (\lambda_{s^-})  \ud N_s^{(1)} }, \quad t \in [0,T].
\end{equation}
In view of the model construction,
we can safely introduce the $(\bF,\Pp)$-compensator measures of $m^\i(\ud t, \ud z), i=1,2$.
\begin{remark}\label{projection}
By the Girsanov Theorem the $(\bF,\Pp)$-predictable projections measures (the so-called compensator measures) of $m^\a(\ud t,\ud z)$ and $m^\b(\ud t,\ud z)$, see \eqref{eqn:m1m2},  are given respectively by
\begin{equation} \label{dualpred}
\nu^\a(\ud t, \ud z)= \lambda_{t^-} F^\a(\ud z) \ud t, \quad \nu^\b(\ud t, \ud z)= \rho F^\b(\ud z) \ud t.
\end{equation}
In particular, $N^\a$ is a point process with $(\bF,\Pp)$-predictable intensity $\{\lambda_{t^-},\ t \in [0,T]\}$, while $N^\b$ remains a point process with constant $(\bF,\Pp)$-intensity $\rho>0$.\\
It turns out that for any $\bF$-predictable nonnegative random field $\{H(t,z),\ t \in [0,T],\ z \in [0, + \infty)\}$ and $i=1,2$
$$
\bE \left[ \int_0^t \int_0^{+\infty} H(s,z)m^\i(\ud s, \ud z) \right] = \bE \left[ \int_0^t \int_0^{+\infty} H(s,z)\nu^\i(\ud s, \ud z) \right], \quad  t \in [0,T],
$$
where $\nu^\i(\ud s, \ud z)$, $i=1,2$, are defined in \eqref{dualpred}.
Moreover, if 
$\bE \left[ \int_0^T \int_0^{+\infty} |H(s,z)|\nu^\i(\ud s, \ud z) \right] <+ \infty,$
then the process
$$ 
\left\{\int_0^t \int_0^{+\infty} H(s,z) \left( m^\i(\ud s, \ud z) - \nu^\i(\ud s, \ud z) \right), \ t \in [0,T]\right\},
$$
is an $(\bF,\Pp)$-martingale.
\end{remark}
\noindent Now, we recall the Markov structure of the intensity $\Lambda$. Equation \eqref{intensity} reads as
\begin{equation}\label{intensity_eq}
\ud \lambda_t = \alpha (\beta - \lambda_t ) \ud t+ \int_0^{+ \infty} \ell(z) m^\a(\ud t , \ud z) + \int_0^{+ \infty}  z m^\b(\ud t , \ud z).
\end{equation}


\begin{proposition}\label{prop:generatore}
The process $\Lambda$ is an $(\bF,\Pp)$-Markov process with generator
\begin{equation}\label{generatore}
{\Ll} f( \lambda) =\alpha (\beta - \lambda) f' (\lambda) +
\int_0^{+ \infty} [ f(\lambda + \ell(z)) - f(\lambda) ] \lambda F^\a (\ud z) \nonumber + \int_0^{+ \infty} [ f(\lambda + z) - f(\lambda) ]  \rho F^\b (\ud z).
\end{equation}
The domain of the generator ${\Ll}$ contains the class of functions $f\in C^{1}(0, + \infty)$ such that
\begin{equation}
 \begin{split}
 &\bE \left[ \int_0^T \int_0^{+ \infty} |f(\lambda_s +\ell(z)) - f(\lambda_s)| \lambda_s  F^\a (\ud z)\ud s \right] <+ \infty, \\
 &\bE \left[ \int_0^T \int_0^{+ \infty}  |f(\lambda_s + z) - f(\lambda_s)|  F^\b (\ud z)\ud s \right] <+ \infty, \quad \bE \left[ \int_0^T \lambda_s |  f' (\lambda_s) | \ud s \right] <+ \infty.
 \end{split}
 \end{equation}

  \end{proposition}

\begin{proof}
It is a direct application of It\^o's formula.
\end{proof}
\noindent We introduce the following assumption, which is needed to prove Proposition \ref{mom} below.
\begin{assumption}\label{nuova1}
\[
 \bE[ (\ell(Z^\a)^k] < + \infty, \quad  \bE[ (Z^\b)^k] < + \infty, \quad \forall k=1,2, \dots.
\]
\end{assumption}
\begin{proposition}\label{mom}
Under Assumption \ref{nuova1}, for any $t \in [0,T]$, 
$\bE \left[ \int_0^t \lambda_s^k \ud s \right] <  + \infty, \quad \forall k=1,2, \dots.$
\end{proposition}
\begin{proof} See \citet[Proposition 2.10]{brachetta_call_ceci_sgarra}.
\end{proof}

\section{Problem formulation}\label{sec:formulation}

We introduce the optimal reinsurance problem from the primary insurer point of view.
The primary insurer aims to subscribe to a reinsurance contract in order to optimally manage her wealth.
The dynamics of the surplus process $R=\{R_t,\ t \in [0,T]\}$ without reinsurance 
is given by
\begin{equation}
\label{eqn:surplus_u=0}
\ud R_t = c_t\,\ud t - \int_0^{+\infty} z \,m^\a(\ud t,\ud z), \qquad R_0\in\mathbb{R}^+,
\end{equation}
where $c=\{c_t,\ t\in[0,T] \}$ denotes the insurance premium, which is assumed to be an $\bF$-predictable process and such that $\mathbb E \left[ \int_0^T c_t \ud t \right]< +\infty$ and $R_0>0$ is the initial capital.
The insurer can choose any reinsurance arrangement in a given class of admissible contracts, which are parametrized by a $n$-uple $u$ (the control) taking values in $U\subseteq\overline{\mathbb{R}}^n$, with $n\in\mathbb{N}$ and $\overline{\mathbb{R}}$ denoting the compactification of $\mathbb{R}$. 
Under an admissible strategy $u\in \mathcal{U}$ (the definition of admissibility set $\mathcal{U}$ will be given in Definition \ref{def:U} below),
it retains the amount $\Phi(Z^\a_j,u_{T^\a_j})$ of the $j$-th claim, while the remaining $Z_j^{(1)} - \Phi(Z^\a_j,u_{T^\a_j})$ is paid by the reinsurer.
Specifically, under an admissible strategy $u = \{ u_t,\ t \in [0,T]\}$ the aggregate losses process covered by the insurer, denoted by $C^u=\{ C^u_t,\ t \in [0,T]\}$, reads as
$$
C^u_t = \sum_{j=1}^{N^\a_t} \Phi(Z^\a_j, u_{T^\a_j}), \quad t \in [0,T],
$$
so that the remaining losses $(C - C^u)$, with $C$ defined by \eqref{loss}, will be undertaken by the reinsurer.
We suppose that $\Phi(z,u)$ the retention function is continuous in $u$ and there exist at least two points $u_N,u_M\in U$ such that
\[
0\le\Phi(z,u_M) \le \Phi(z,u) \le \Phi(z,u_N) = z \qquad \forall (z,u)\in[0,+\infty)\times U,
\]
so that $u=u_N$ corresponds to null reinsurance, while $u=u_M$ represents the maximum reinsurance protection. Note that $u_M$ corresponds to full reinsurance when applicable.

\begin{example}\label{ex_reinsurance}
For the reader's convenience, we recall \citet[Example 4.2]{brachetta_call_ceci_sgarra}, which shows how
standard reinsurance contracts fit our modeling framework.
\begin{enumerate}
\item Under {\em proportional reinsurance}, the insurer transfers a percentage $(1-u)$ of any future loss to the reinsurer, so we set
\[
\Phi (z,u) = uz, \qquad u \in [0,1].
\]
Selecting the scalar $u\in[0,1]=:U$ is equivalent to choosing the retention level of the contract. Notice that here $u_N=1$ means no reinsurance and $u_M=0$ corresponding to full reinsurance.
\item Under an {\em excess-of-loss reinsurance} policy, the reinsurer covers all the losses exceeding a retention level $u$, hence we fix the class of all the functions with this form:
\[
\Phi (z,u) = u \wedge z, \qquad u \in [0, +\infty ].
\]
So, here $U:=[0,+\infty]$, $u_N=+\infty$ and $u_M=0$ corresponds, to full reinsurance. 
\item Under a {\em limited excess of loss reinsurance}, for any claim the reinsurer covers the losses exceeding a threshold $u_1$, up to a maximum level $u_2>u_1$, so that the maximum loss is limited to $(u_2-u_1)$ on the reinsurer's side. In this case:
\[
\Phi (z,u) = z-  (z- u_1 )^{+} + (z- u_2 )^{+},
\]
so that $U=\{ (u_1,u_2): u_1\ge0, u_2\in[u_1,+\infty] \}$ and $u=(u_1,u_2)$. Clearly, we have that $u_M = (u_{M,1}, u_{M,2})=(0,+\infty)$ and $u_N$ can be any point on the line $u_1=u_2$.
A special case is the so-called {\em limited excess of loss with fixed reinsurance coverage}, in which $u_2 = u_1 + \beta_M$, with $\beta_M >0$. Here, $U=[0,+\infty]$, $u_N=+ \infty$ and $u_M = 0$ corresponds to the maximum reinsurance coverage $\beta_M$. For $\beta_M = + \infty$ this case reduces to the excess-of-loss reinsurance.
\end{enumerate}
\end{example}
\noindent The insurer will have to pay a reinsurance premium $q^u=\{q^u_t,\ t\in[0,T] \}$, which depends on the strategy $u$, satisfying
the following assumptions.
\begin{assumption}
 The reinsurance premium  admits the following representation:
\begin{equation}\label{eqn:q_of_u}
q^u_t (\omega) = q (t, \omega , u ) \quad \forall (t, \omega , u ) \in [0,T] \times \Omega \times U,
\end{equation}
for a given function $q (t, \omega , u )\colon[0,T] \times \Omega \times U \rightarrow [0,+\infty) $ continuous in $u$, $\bF$-predictable and with continuous partial derivatives $\ds \frac{\partial q (t, \omega , u )}{\partial u_i} $, $i=1,\dots,n$.
Moreover, for any $(t,\omega) \in [0,T] \times \Omega$, 
\[
q (t, \omega , u_N )= 0, \quad q (t, \omega , u ) \leq q (t, \omega , u_M ), \quad \forall u\in U,
\]
since a null protection is not expensive and the maximum reinsurance is the most expensive. 
\end{assumption}
\noindent In the following $q^u$ will denote the reinsurance premium associated with the dynamic reinsurance strategy $\{u_t,\ t \in [0,T] \}$. Notice that both insurance and reinsurance premiums are assumed to be $\bF$-predictable, since the insurer and the reinsurer share the same information.  Finally, we require the following integrability condition:
$$ \mathbb{E}\Big[ \int_0^T q^{u_M}_t \ud t \Big] < +\infty,$$
which ensures that for any $u \in\mathcal{U}$, we have
$
\mathbb{E} \left[ \int_{0}^{T} q^u_s ds \right] < +\infty.
$

\noindent Summarizing, the surplus process $R^u=\{R_t^u,\ t \in [0,T]\}$ with reinsurance evolves according to
\begin{equation}
\label{eqn:surplus}
dR^u_t = \bigl(c_t-q^u_t\bigl)\,\ud t - \ud C^u_t = \left( c_t-q^u_t \right) \ud t - \int_0^{+\infty} \Phi(z, u_t) \, m^\a(\ud t,\ud z), \qquad R^u_0= R_0\in\mathbb{R}^+.
\end{equation}
Moreover, the insurer invests her surplus in a risk-free asset with constant interest rate $r \in \R^+$, so that for any reinsurance strategy $u\in\mathcal{U}$ the wealth $X^u=\{X_t^u,\ t \in [0,T]\}$ satisfies the following SDE
\begin{equation}
\label{eqn:X}
\ud X^u_t = \ud R^u_t + rX^u_t\,\ud t, \qquad X^u_0= R_0\in\mathbb{R}^+ ,
\end{equation}
whose solution is given by
\begin{equation}
\label{eqn:X_explicit}
X^u_t = R_0e^{rt} + \int_0^t e^{r(t-s)} \left( c_s-q^u_s \right) \,\ud s
-\int_0^t\int_0^{+\infty} e^{r(t-s)} \Phi(z, u_s)  \,m^{(1)}(\ud s,\ud z), \quad t \in [0,T].
\end{equation}
\begin{remark}
Notice that the stochastic wealth $X^u$ can possibly take negative values, due to the possibility of borrowing money from the bank account.
\end{remark}

As mentioned earlier, the insurer aims at optimally controlling her wealth using reinsurance. More formally, she aims at maximizing the expected exponential utility of terminal wealth over the class $\mathcal U$, that is,
\[
\sup_{u\in\mathcal{U}}\mathbb{E}\bigl[ 1-e^{-\eta X^u_T} \bigr],
\]
which turns out trivially to be equivalent to the minimization problem
\begin{equation}\label{eqn:minpb}
\inf_{u\in\mathcal{U}}\mathbb{E}\bigl[ e^{-\eta X^u_T} \bigr],
\end{equation}
where $\eta \in \R^+$ denotes the insurer risk aversion.

\begin{definition}
\label{def:U}
We denote by $\mathcal{U}$ the class of admissible strategies, which are all the $U$-valued and predictable processes, $\{u_t, t\in[0,T]\}$,  such that $\mathbb{E}\bigl[ e^{-\eta X^u_T} \bigr] < +\infty$.
Given $t \in [0,T]$, we will denote by $\mathcal{U}_t$ the class $\mathcal{U}$ restricted to the time interval $[t,T]$.
\end{definition}
\noindent The next assumptions are required in the sequel.
\begin{assumption}\label{ass_app_premium}
	We assume that for every $a >0$
	\begin{itemize}
	\item[i)]
	$ \mathbb E \left[ e^{a \ell(Z^\a)} \right] < +\infty, \quad \mathbb E  \left[ e^{a Z^\a} \right] < +\infty, \quad \mathbb E  \left[ e^{a Z^\b} \right] < +\infty;$
	\item[ii)]
	$
	\mathbb E\left[ e^{a \int_0^T q^{u_M}_t\,\ud t} \right] < +\infty.
	$
	\end{itemize}
\end{assumption}

\begin{proposition}\label{ADM}
    Under Assumption \ref{ass_app_premium},  we have that 
    \begin{itemize}
	\item[(a)]
 for every $a>0$, $\mathbb{E}[e^{a C_T}] < +\infty$ and  
 $\mathbb{E}[e^{a \int_0^T \lambda_s \ud s}] < +\infty$;
 
 \item[(b)] any $U$-valued and $\bF$-predictable process is admissible according to Definition \ref{def:U}.
\end{itemize}
 
\end{proposition}

\begin{proof}
(a) It is proved in \citet[Lemma 4.6 and Lemma B.1]{brachetta_call_ceci_sgarra}.

(b) Since for each $u \in \mathcal U$, $t \in [0,T]$, $q_t^u \leq q_t^{u_M}$, and using the inequality $(a+b)^2 \leq \frac{1}{2}(a^2 + b^2)$, $a,b \in \R$, 
 we have for any $u \in \mathcal U$
\begin{align}
  \esp{ e^{-\eta X^u_T}}&=\esp{e^{-\eta R_0e^{rT}}e^{-\eta\int_0^T e^{r(T-s)} \left( c_s-q^u_s \right) \,\ud s}
e^{\eta\int_0^T\int_0^{+\infty} e^{r(T-s)} \Phi(z, u_s)  \,m^{(1)}(\ud s,\ud z)}}\\
& \leq \esp{e^{\eta\int_0^T e^{r(T-s)} q^{u_M}_s \,\ud s}
e^{\eta e^{r T}\int_0^T\int_0^{+\infty}  z  \,m^{(1)}(\ud s,\ud z)}}\\
& \leq \frac{1}{2}\left(\esp{e^{2\eta e^{rT} \int_0^T q^{u_M}_s \,\ud s}}+\esp{e^{2\eta e^{r T}C_T}}\right),
\end{align}
which is finite in view of Assumption \ref{ass_app_premium}  and (a).

\end{proof}

\begin{remark}
Insurance companies usually apply a maximum policy $D>0$, i.e., they only repay claims up to the amount $D$ to the policyholders. In this setting, claims' sizes are of the form $\min\{Z^\a_n, D\} \leq D$, hence condition $\mathbb E  \left[ e^{a Z^\a} \right] <+\infty$ in Assumption \ref{ass_app_premium} is trivially satisfied.
\end{remark}
\noindent We conclude the section by presenting the most commonly used premium principles.

\begin{example}[Premium principles] \label{ex_EVP}
Under any admissible reinsurance strategy $u\in\mathcal{U}$, the expected cumulative losses covered by the reinsurer in the interval $[0,t]$, with $t \leq T$, are given by
\[
 \mathbb{E}\left[ \int_0^t \int_0^{+\infty} ( z - \Phi(z,u_s)) \, m^\a( \ud s , \ud z) \right] = \mathbb{E}\left[ \int_0^t \int_0^{+\infty} ( z - \Phi(z,u_s)) \, \lambda_{s^-} F^\a(\ud z) \ud s \right].
 \]
 (i) According to the {\em expected value principle (EVP)}, the insurance premium $c$ is given by
 \begin{equation} 
c_t = (1+ \theta_I ) \lambda_{t^-}  \int_0^{+\infty}   z  F^\a(\ud z),
\end{equation}
where $\theta_I>0$ denotes the safety loading applied by the insurer, and
 the reinsurance premium $q^u$ has to satisfy
\[
 \begin{split}
\mathbb{E}\left[ \int_0^t  q_s^u \, \ud s \right] &= (1 + \theta_R) \mathbb{E}\left[ \int_0^t \int_0^{+\infty} ( z - \Phi(z,u_s)) \, \lambda_{s^-} F^\a(\ud z) \ud s \right], \quad \forall u\in\mathcal{U}, \ \forall t \in [0,T],
\end{split}
\]
where $\theta_R>0$ denotes the safety loading applied by reinsurer.
Thus,
\begin{equation}\label{ex_EVP1}
    q_t^u = (1+ \theta _R) \lambda_{t^-} \int_0^{+\infty}   \left(  z - \Phi(z,u_t) \right) F^\a(\ud z).
    \end{equation}
(ii) Under the {\em variance premium principle (VPP)}, the insurance and reinsurance premiums are given by
\begin{equation}\label{VP}
\begin{split}
& c_t = \lambda_{t^-} \left\{\int_0^{+\infty} z F^\a(\ud z)  +\eta_I \int_0^{+\infty}  z^2 F^\a(\ud z) \right \}, \\
    & q_t^u = \lambda_{t^-} \left\{ \int_0^{+\infty}   \left(  z - \Phi(z,u_t) \right) F^\a(\ud z) +\eta_R \int_0^{+\infty}   \left(  z - \Phi(z,u_t) \right)^2 F^\a(\ud z) \right \},
    \end{split}
    \end{equation}
    respectively, where $\eta_I>0$ and $\eta_R>0$ are the variance loadings applied by insurer and reinsurer, respectively.\\
(iii) Recently, a more general premium has been considered in the literature (see e.g. \citet{cao2023stackelberg}), the {\em mean-variance principle (MVP)} for which both the mean and variance loadings are allowed to depend on both time and loss 
    \begin{equation}\label{MVP}
    \begin{split}
& c_t = \lambda_{t^-} \int_0^{+\infty} \Big \{ (1 +  \theta _I(t,z)) z + \eta_I (t, z) z^2 \Big \} F^\a(\ud z), \\
    &q_t^u = \lambda_{t^-} \int_0^{+\infty} \Big \{ (1 +  \theta _R(t,z)) \left(  z - \Phi(z,u_t) \right) + \eta_R (t, z) \left(  z - \Phi(z,u_t) \right)^2 \Big \} F^\a(\ud z),
    \end{split}
    \end{equation}
    for some non-negative loading factors $\theta _I(t,z)$ and $\eta_I (t, z)$, $\theta _R(t,z)$ and $\eta_R (t, z)$, applied by the insurer and the reinsurer, respectively. 

\end{example}

\section{The Hamilton-Jacobi-Bellman equation}\label{HJB section}
\noindent In the sequel, we work in a Markovian setting, then making the following assumption.
\begin{assumption}\label{ass:markovian}
    The premium rates for insurance and reinsurance are given by $c_t=c(t,\lambda_{t^-})$ and $q^u_t=q(t,\lambda_{t^-},u_t)$ where  $c: [0,T] \times (0,+\infty) \to (0,+\infty)$ and $q: [0,T] \times (0,+\infty) \times U \to [0,+\infty)$ are measurable functions.  
\end{assumption}

\begin{remark}
    Assumption \ref{ass:markovian} is satisfied for the premium principles  described in Example \ref{ex_EVP}.
\end{remark}
\noindent Then, the wealth $X^u$ of the insurer associated with the reinsurance strategy $u$ evolves as
\begin{align}\label{SDE_X}
	\ud X_t^u 
	&=\left[c(t,\lambda_t) - q(t,\lambda_t,u_t)+rX_t^u\right]\ud t- \!\!\int_0^{+\infty}\!\!\! \! \Phi(z,u_t)m^{(1)}(\ud t,\ud z),\  X_0^u=R_0 \in \R^+.
\end{align}

\noindent Note that under Assumption \ref{ass:markovian} the pair $(X^u, \Lambda)$, for any $u \in U$, is a Markov process. 
\begin{definition}
The set $\mathcal D$ denotes the class of functions $f \in C^1([0,T]\times \R \times (0,+\infty))$ such that for every constant $u \in U$ we have
\begin{equation}
 \begin{split}
 &\bE \left[ \int_0^T \int_0^{+ \infty} |f(s,X_s^u-\Phi(z,u),\lambda_s +\ell(z)) - f(s,X_s^u,\lambda_s)| \lambda_s  F^\a (\ud z)\ud s \right] < +\infty, \\
 &\bE \left[ \int_0^T \int_0^{+ \infty}  |f(s,X_s^u,\lambda_s + z)) - f(s,X_s^u,\lambda_s)|  \rho F^\b (\ud z)\ud s \right] < +\infty,
 \end{split}
 \end{equation}

  \begin{align}
   & \bE \left[ \int_0^T \lambda_s \Big{|}\frac{\partial f}{\partial \lambda} (s,X_s^u,\lambda_s)\Big{|} \ud s \right]  < +\infty, \quad \bE \left[ \int_0^T  \left|c(s,\lambda_s) - q(s,\lambda_s,u)+r X_s^u\right| \Big{|}\frac{\partial f}{\partial x} (s,X_s^u,\lambda_s)\Big{|} \ud s \right]  < +\infty.
  \end{align}
\end{definition}
\begin{lemma}
The Markov generator of the  
the pair $(X^u, \Lambda)$ for all constant controls $u \in U$ is given by
\begin{align}
&{\Ll}^{X,\lambda,u} f(t, x, \lambda) = \frac{\partial f}{\partial t}(t, x, \lambda) +  \frac{\partial f}{\partial x}(t, x, \lambda) \left\{c(s,\lambda) - q(s, \lambda,u)+rx\right\}+\frac{\partial f}{\partial \lambda}(t, x, \lambda) \alpha (\beta -\lambda)\\
& \quad + \int_0^{+\infty}\left[f(t,x-\Phi(z,u), \lambda + \ell(z))-f(t,x,\lambda)\right] \lambda F^\a (\ud z)\\
& \quad + \int_0^{+ \infty} \left[ f(t,x,\lambda + z) - f(t,x,\lambda) \right]  \rho F^\b (\ud z).
\end{align}
The domain of the generator ${\Ll}^{X,\lambda,u}$ contains $\mathcal D$.
\end{lemma}

\begin{proof} Recalling \eqref{intensity_eq} and \eqref{SDE_X}, the proof is a direct  application of It\^o's formula. 
\end{proof}

 Let us introduce the value function $v:[0,T] \times \R \times (0,+\infty) \to [0,+\infty)$ corresponding to the optimization problem \eqref{eqn:minpb}
\begin{equation}
\label{VF}
v(t,x,\lambda) = \inf_{u\in\mathcal{U}_t}\mathbb{E}_{t,x,\lambda}\bigl[ e^{-\eta X^u_T} \bigr], \quad (t,x,\lambda) \in [0,T] \times \R \times (0,+\infty),
\end{equation}
where the notation $\mathbb{E}_{t,x,\lambda}[\cdot]$ stands for the expectation when $(X^u, \Lambda)$ starts from $(x,\lambda)$ at time $t$. If the value function $v(t,x,\lambda)$ is sufficiently smooth, it is expected to solve the Hamilton-Jacobi-Bellman (HJB) equation
\begin{equation}
\left\{
\begin{array}{ll}
\inf_{u \in U} {\Ll}^{X,\lambda,u} v(t,x, \lambda) =0, &\quad \forall (t,x,\lambda) \in (0,T) \times \R \times (0,+\infty), \label{HJB1}\\
v(T,x, \lambda) = e ^{-\eta x}, &\quad \forall (x,\lambda) \in \R \times (0,+\infty).
\end{array}
\right.
\end{equation}

\noindent Notice that the explicit solution to SDE \eqref{SDE_X} with initial data $(t,x)$, is given by
\begin{equation}
\label{eqn:X_t_T}
X^u_T = xe^{r(T-t)} + \int_t^T e^{r(T-s)} \left( c(s,\lambda_s)-q(s, \lambda_s,u_s) \right) \,\ud s
-\int_t^T\int_0^{+\infty} e^{r(T-s)} \Phi(z, u_s)  \,m^{(1)}(\ud s,\ud z),
\end{equation}
so we can write \eqref{VF} as follows
\begin{equation}\label{eq:VF2}
    v(t,x,\lambda)=e^{-\eta x e^{r(T-t)}}\varphi(t,\lambda),
\end{equation}
where we have set for any $(t, \lambda) \in [0,T] \times (0,+\infty)$
\begin{equation} \label{varphi}
\varphi(t,\lambda)=\inf_{u\in\mathcal{U}_t}\mathbb{E}_{t,\lambda}\left[e^{-\eta\int_t^T e^{r(T-s)} \left( c(s,\lambda_s)-q(s,\lambda_s,u_s) \right) \,\ud s
+\eta\int_t^T\int_0^{+\infty} e^{r(T-s)} \Phi(z, u_s)  \,m^{(1)}(\ud s,\ud z)}\right],
\end{equation}
and the notation $\mathbb{E}_{t,\lambda}[\cdot]$ stands for the expectation computed when the intensity process $\Lambda$ starts from $\lambda$ at time $t$.

\begin{remark} \label{positive}
    It is easy to prove that $\varphi(t,\lambda)$ is a strictly positive  function. Indeed, for any $u\in\mathcal{U}_t$ we have
$$\mathbb{E}_{t,\lambda}\left[e^{-\eta\int_t^T e^{r(T-s)} \left( c(s,\lambda_s)-q(s,\lambda_s,u_s) \right) \,\ud s
+\eta\int_t^T\int_0^{+\infty} e^{r(T-s)} \Phi(z, u_s)  \,m^{(1)}(\ud s,\ud z)}\right]
\geq \mathbb{E}_{t,\lambda}\left[e^{-\eta\int_t^T e^{r(T-s)}  c(s, \lambda_s)\,\ud s} \right ].$$
Thus, for any $(t, \lambda) \in [0,T] \times (0,+\infty)$, we get $\varphi(t,\lambda) \geq  \mathbb{E}_{t,\lambda}\left[e^{-\eta\int_t^T e^{r(T-s)}  c(s,\lambda_s)\,\ud s} \right]>0$. 
\end{remark}
\noindent Now, we compute all partial derivatives of $v(t,x,\lambda)$:
\begin{align}
\frac{\partial v}{\partial t}(t,x,\lambda) & = e^{-\eta x e^{r(T-t)}}\left(\frac{\partial \varphi}{\partial t}(t,\lambda) + r \eta x e^{r(T-t)}\varphi(t,\lambda)\right),	\\
\frac{\partial v}{\partial x}(t,x,\lambda) & = -\eta e^{r(T-t)} e^{-\eta x e^{r(T-t)}} \varphi(t, \lambda),\\
\frac{\partial v}{\partial \lambda}(t,x,\lambda) & =  e^{-\eta x e^{r(T-t)}} \frac{\partial \varphi}{\partial \lambda}(t, \lambda).
\end{align}
Then, 
\begin{align}
& {\Ll}^{X,\lambda,u} v(t, x, \lambda) \\
& =  e^{-\eta x e^{r(T-t)}}\bigg\{\left(\frac{\partial \varphi}{\partial t}(t,\lambda) + r \eta x e^{r(T-t) }\varphi(t,\lambda)\right)-\eta e^{r(T-t)}\varphi(t, \lambda)\left(c(t,\lambda)-q(t,\lambda,u)+rx\right)\\
& \  + \frac{\partial \varphi}{\partial \lambda}(t, \lambda)\alpha (\beta -\lambda)\bigg\}+ \int_0^{+\infty}e^{-\eta x e^{r(T-t)}}\left[e^{\eta \Phi(z, u)  e^{r(T-t)}}\varphi(t, \lambda + \ell(z))-\varphi(t,\lambda)\right] \lambda F^\a (\ud z)\\
& \  + \int_0^{+ \infty} e^{-\eta x e^{r(T-t)}}\left[ \varphi(t,\lambda + z) - \varphi(t,\lambda) \right]  \rho F^\b (\ud z)\\
& = e^{-\eta x e^{r(T-t)}}\bigg\{\frac{\partial \varphi}{\partial t}(t,\lambda) - \eta e^{r(T-t)}\varphi(t,\lambda)\left(c(t,\lambda) -q(t,\lambda,u)\right) + \frac{\partial \varphi}{\partial \lambda}(t, \lambda)\alpha (\beta -\lambda)\\
& \ + \int_0^{+\infty}\left[e^{\eta \Phi(z, u)  e^{r(T-t)}}\varphi(t, \lambda + \ell(z))-\varphi(t,\lambda)\right] \lambda F^\a (\ud z)+ \int_0^{+ \infty} \left[ \varphi(t,\lambda + z) - \varphi(t,\lambda) \right]  \rho F^\b (\ud z)\bigg\},
\end{align}
and the HJB-equation \eqref{HJB1} reduces to
\begin{equation} \label{eq:HJB2}
\begin{split}
\frac{\partial \varphi}{\partial t}(t,\lambda) & + \alpha(\beta -\lambda)	\frac{\partial \varphi}{\partial \lambda}(t,\lambda) + \int_0^{+ \infty} \left[ \varphi(t,\lambda + z) - \varphi(t,\lambda) \right]  \rho F^\b (\ud z)\\
& - \eta e^{r(T-t)} \varphi(t,\lambda) c(t,\lambda) +\inf_{u \in U} \Psi^u(t,\lambda)=0,
\end{split}
\end{equation}
with final condition
\begin{equation} \label{eq:final_cond}
\varphi(T,\lambda)=1, \quad \lambda \in (0,+\infty),
\end{equation}
where the function $\Psi^u(t,\lambda)$ is given by
\begin{equation}\label{Psiu}
	\Psi^u(t,\lambda)  = \eta e^{r(T-t)} \varphi(t,\lambda) q(t, \lambda,u)
	+\int_0^{+\infty}\left[e^{\eta \Phi(z, u) e^{r(T-t)}}\varphi(t,\lambda + \ell(z))-\varphi(t,\lambda)\right]  \lambda F^\a (\ud z).
\end{equation}

\noindent We now provide a verification result in terms of the reduced HJB-equation \eqref{eq:HJB2}, whose proof is postponed to Appendix \ref{appendix:proofs}.
\begin{theorem}[Verification Theorem]\label{Verifica}
\hskip 1mm

(i) Let $\widetilde \varphi (t,\lambda)\in C^1((0,T) \times (0,+\infty)) \cap C([0,T] \times(0,+\infty))$ be a classical solution of the HJB-equation \eqref{eq:HJB2} that meets the final condition \eqref{eq:final_cond}. 

(ii) Let $\widetilde v(t,x,\lambda) = e^{-\eta x e^{r(T-t)}} \widetilde\varphi(t,\lambda)$ and assume that for any $u \in\mathcal U$ the family $\{\widetilde v(\tau, X^u_{\tau},\lambda_\tau);\ \tau \  \mbox{stopping time}, \ \tau\leq T \}$ is uniformly integrable. \\
\noindent Then, the function $\widetilde v(t,x,\lambda)$ coincides with the value function $v(t,x,\lambda)$ given in \eqref{VF}. Furthermore, let $u^*(t,\lambda)$ be a minimizer of $\inf_{u \in U} \Psi^u(t,\lambda)$. Then, $\{u_t^*=u^*(t,\lambda_{t^{-}}),\ t \in [0,T]\} \in \mathcal U$ is an optimal strategy. 
\end{theorem}

\noindent In the one-dimensional case the function $\Phi(z,u)$ is increasing in $u \in [u_M, u_N] \subset \overline{ \mathbb R}$ and under suitable assumptions we can obtain an explicit representation of the optimal reinsurance strategy.
\begin{proposition}[The optimal strategy]\label{strategyHJB}
   Under the assumptions of Theorem \ref{Verifica}, suppose that $\Phi (z, u)$ is differentiable in $u\in [u_M,u_N]$ for almost every $z \in (0, + \infty)$ and $\Psi^u(t,\lambda)$ given in \eqref{Psiu} is strictly convex in $u\in [u_M,u_N]$. Then, the optimal reinsurance strategy $u^*_t=\{ u^*(t,\lambda_{t^-}),\ t\in[0,T] \}$ is given by
\begin{equation}\label{u*_general_HJB}
u^*(t,\lambda_{t^-}) =
\begin{cases}
u_M		& (t,\lambda_{t^-}) \in A_0 \\
u_N			& (t,\lambda_{t^-})\in A_1\\
\bar{u}(t,\lambda_{t^-})	& \text{otherwise}, \\
\end{cases}
\end{equation}
where
\begin{align*}
A_0 &= \left\{ (t,\lambda) \in [0,T] \times (0, + \infty) :\ h(t, \lambda, u_M ) < 0 \right\}, \\
 A_1 &=  \left\{ (t,\lambda) \in [0,T] \times (0, + \infty) :\ h(t, \lambda, u_N ) > 0 \right\},
\end{align*}
\begin{equation}\label{eq:h}
h(t, \lambda,u) =  - \varphi(t,\lambda) \frac{\partial q}{\partial u}(t,\lambda,u) - \int_{0}^{\infty} \varphi(t, \lambda + \ell(z))  e^{\eta e^{r (T-t)}\Phi (z, u)}  \frac{\partial \Phi}{\partial u}(z,u) \lambda F^{(1)} (\ud z),
\end{equation}
the function $\varphi(t,\lambda)$ is given in \eqref{varphi} and $\bar{u}(t,\lambda) \in (u_M,u_N)$ solves the following equation with respect to $u$:
\begin{equation}\label{1st_order_genHJB}
- \varphi(t,\lambda) \frac{\partial q}{\partial u}(t, \lambda,u) = \int_{0}^{\infty} \varphi(t, \lambda + \ell(z))  e^{\eta e^{r (T-t)}\Phi (z, u)}  \frac{\partial \Phi}{\partial u}(z,u) \lambda F^{(1)} (\ud z).
\end{equation}

\end{proposition}

\begin{proof}
By the Verification Theorem (see Theorem \ref{Verifica})  we have that $\varphi(t, \lambda)$ given in \eqref{varphi} solves the reduced HJB-equation \eqref{eq:HJB2} with final condition \eqref{eq:final_cond}
and in order to obtain an optimal strategy we have to find a minimizer of $\inf_{u \in U} \Psi^u(t,\lambda)$. Since $\Psi^u(t,\lambda)$ given in \eqref{Psiu} is continuous and strictly convex in $u\in [u_M,u_N]$, we see that there exists a unique minimizer. We write down the first order condition 
$$\frac {\partial \Psi^u}{\partial u}(t,\lambda)=- \eta e^{r(T-t)} h(t,\lambda, u)=0,
$$ 
where the function $h(t,\lambda, u)$ is defined in \eqref{eq:h}.
Recalling that $\ds \frac {\partial \Psi^u}{\partial u}(t,\lambda)$ is an increasing function in $u \in [u_M, u_N]$ we have three possible cases:

\noindent (i) for $(t,\lambda) \in A_0$, $\ds \frac {\partial \Psi^u}{\partial u}(t,\lambda) > 0 $, so $\Psi^u(t, \lambda)$ is increasing in $u\in  [u_M, u_N]$, hence $u^*(t, \lambda) = u_M$;

\noindent (ii) for $(t,\lambda) \in A_1$, $\ds \frac {\partial \Psi^u}{\partial u}(t,\lambda) < 0 $, so $\Psi^u(t, \lambda)$ is decreasing in $u \in  [u_M, u_N]$, hence $u^*(t, \lambda) = u_N$;

\noindent (iii) otherwise there exists
$\bar{u}(t,\lambda) \in (u_M,u_N)$ such that $\ds \frac{\partial \Psi^u}{\partial u}(t,\lambda)\Big{|}_{u=\bar u}=0$ and
$u^*(t, \lambda) = \bar u(t, \lambda)$. 
Finally, from Theorem \ref{Verifica} the stochastic process $\{u_t^*=u^*(t,\lambda_{t^{-}}),\ t \in [0,T]\}$ provides an optimal strategy. 
\end{proof}

    \begin{remark}
        If $q(t,\lambda,u)$ and $\Phi(z,u)$ are linear or convex in $u\in[u_M,u_N]$, then $\Psi^u(t,\lambda)$ is strictly convex in $u\in[u_M,u_N]$, and Proposition \ref{strategyHJB} applies.
    \end{remark}
The method based on the Verification Theorem in the contagion model outlined in Section \ref{sec:model} is challenging due to the inherent complexity in establishing the existence of a classical solution to Equation \eqref{eq:HJB2}. 
This difficulty arises from the nature of \eqref{eq:HJB2} as a Partial Integro-Differential Equation (PIDE) coupled with an optimization component embedded within the associated integro-differential operator.
For this reason we develop in Section \ref{sec:bsde} an alternative approach based on BSDEs. This method enables us  to obtain the result given in Proposition \ref{strategyHJB} without imposing any requirements for regularity on the value function $\varphi(t, \lambda)$, see Proposition \ref{prop:optreins_general} below.\\
Now, we discuss the special case of  the Cox process with shot noise intensity, which does not exhibit the self-exciting effect.

\subsection{Cox process with shot noise intensity}

Notice that the case of the Cox process with shot noise intensity corresponds to $\ell(z)=0$ in \eqref{intensity} and the reduced HJB-equation \eqref {eq:HJB2} reads as 
\begin{equation} \label{eq:HJB4}
\begin{split}
& \frac{\partial \varphi}{\partial t}(t,\lambda) + \alpha(\beta -\lambda)	\frac{\partial \varphi}{\partial \lambda}(t,\lambda) + \int_0^{+ \infty} \left[ \varphi(t,\lambda + z) - \varphi(t,\lambda) \right]  \rho F^\b (\ud z) - \eta e^{r(T-t)} \varphi(t,\lambda) c(t,\lambda)\\
& \quad + \varphi(t,\lambda) \inf_{u \in U} \Big\{ \eta e^{r(T-t)} q(t,\lambda,u) + \int_0^{+\infty}\left( e^{\eta \Phi(z, u) e^{r(T-t)}}- 1\right)  \lambda F^\a (\ud z)\Big \}=0,
\end{split}
\end{equation}
with final condition
\begin{equation} \label{eq:final_cond_cox}
\varphi(T,\lambda)=1, \quad \lambda \in (0,+\infty).
\end{equation}
By continuity of $q(\lambda,u)$ with respect  to $u \in U$ and since $U$ is compact, as in the general case, there exists $u^*(t,\lambda)$ which realizes the infimum in \eqref{eq:HJB4}. Denoting by 
$$H^*(t,\lambda) = \inf_{u \in U} \Big\{ \eta e^{r(T-t)} q(t,\lambda,u) + \int_0^{+\infty}\left( e^{\eta \Phi(z, u) e^{r(T-t)}}- 1\right)  \lambda F^\a (\ud z)\Big \},$$
equation \eqref{eq:HJB4} reads as 
\begin{equation} \label{eq:HJB3}
\begin{split}
\frac{\partial \varphi}{\partial t}(t,\lambda) & + \alpha(\beta -\lambda)	\frac{\partial \varphi}{\partial \lambda}(t,\lambda) + \int_0^{+ \infty} \left[ \varphi(t,\lambda + z) - \varphi(t,\lambda) \right]  \rho F^\b (\ud z)\\
& - \varphi(t,\lambda) \left[ \eta e^{r(T-t)} c(t,\lambda) - H^*(t,\lambda) \right]=0.
\end{split}
\end{equation}
If $\varphi(t,\lambda)$ is a classical solution to  \eqref{eq:HJB3} we can apply the Feymnan-Kac formula to get the probabilistic representation 
\begin{equation} \label{F-K}
\varphi(t,\lambda)=\mathbb E_{t,\lambda}\left[e^{-\int_t^T\left(\eta e^{r(T-t)}c(s,\lambda_s) - H^*(s,\lambda_s)\right)\ud s}\right],
\end{equation}
where the process $\Lambda=\{\lambda_t,\ t \in [0,T]\}$ solves \eqref{intensity_eq} with $\ell(z)=0$.
Note that, due to the fact that the claim intensity $\Lambda$ is an unbounded process, there are no results that we can apply directly to prove existence of classical solutions to \eqref{eq:HJB3} or equivalently to prove that the function $\varphi(t,\lambda)$ given in \eqref{F-K} is sufficiently regular. 

We remark that, 
under the assumptions of Proposition \ref{strategyHJB} we get the regions $A_0$, $A_1$ and the optimal strategy do not depend explicitly on the function  $\varphi(t,\lambda)$.  Precisely, \eqref{1st_order_genHJB}
reads as
\begin{equation}\label{Cox_HJB}
- \frac{\partial q}{\partial u}(t,\lambda,u) = \int_{0}^{\infty}  e^{\eta e^{r (T-t)}\Phi (z, u)}  \frac{\partial \Phi}{\partial u}(z,u) \lambda F^{(1)} (\ud z).
\end{equation}


\section{A BSDE approach} \label{sec:bsde}

\noindent In this section we follow an alternative method based on backward stochastic differential equations (BSDEs). From now on we assume that 
\begin{equation}
 \bF = \bF^{m^\a}   \vee \bF^{m^\b}, \quad \F = \F_T,
\end{equation}
where $\bF^{m^{(i)}} =\{ \F^{m^{(i)}}_t,\ t\in [0,T]\}$ denotes the natural filtration generated by the integer-valued random measures $m^{(i)}(\ud t, \ud z)$, $i=1,2$ given in \eqref{eqn:m1m2}, that is
for any $t \in [0,T]$
$$\F^{m^{(i)}}_t = \sigma\{(T^{(i)}_n, Z^{(i)}_n) : \  t \leq T^{(i)}_n \}.$$
Let us introduce the Snell envelope associated to the stochastic control problem in \eqref{eqn:minpb}, defined for any $u\in\mathcal{U}$ as
\begin{equation}
\label{eqn:W^u}
W^u_t = \essinf_{\bar{u}\in\mathcal{U}(t,u)}
{\mathbb{E}\biggl[e^{-\eta X^{\bar{u}}_T}\Big{|} \mathcal{F}_t\biggr]},\quad  \forall t \in [0,T],
\end{equation}
where $\mathcal{U}(t,u)$ denotes the restricted class of controls almost surely equal to $u$ over $[0,t]$, i.e.
\begin{equation*}
\mathcal U(t,u):=\Big\{ \bar u \in \mathcal U: \bar{u}_s = u_s \ \Pp-\text{a.s.} \ \text{for all} \ s\le t \le T  \Big\}.
\end{equation*}
Denote by $\bar{X}^u_t=e^{-rt}X^u_t$, with $t \in [0,T]$, the discounted wealth:
\begin{equation}
\label{eqn:X_disc}
\bar X^u_t = R_0 + \int_0^t e^{-rs}\left( c_s-q^u_s \right) \,\ud s
-\int_0^t\int_0^{+\infty} e^{-rs} \Phi(z, u_s) \,m^\a(\ud s,\ud z),
\end{equation}
and introduce the value process $V=\{V_t,\ t \in [0,T]\}$ associated to the problem in \eqref{eqn:minpb} as follows,
\begin{equation}
\label{eqn:V}
V_t = \essinf_{\bar{u}\in\mathcal{U}_t}\condespf{e^{-\eta e^{rT}(\bar{X}^{\bar{u}}_T-\bar{X}^{\bar{u}}_t)}},\quad  \forall t \in [0,T];
\end{equation}
then, we can show that, for every $u \in \mathcal{U}$
\begin{equation}\label{eqn:W^u_and_V}
W^u_t = e^{-\eta \bar{X}^u_t e^{rT}} V_t ,
\end{equation}
and, in turn, choosing null reinsurance, i.e. $u_t=u_N$, for any $t \in [0,T]$, we get
\begin{equation}
\label{eqn:VisWI}
V_t = e^{\eta \bar{X}^N_t e^{rT}} W^N_t, 
\end{equation}
where $\bar{X}^N=\{\bar{X}^N_t,\ t \in [0,T] \}$ and $W^N=\{W^N_t,\ t \in [0,T] \}$ denote the discounted wealth and the Snell envelope associated to null reinsurance, given in equations \eqref{eqn:X_disc} and \eqref{eqn:W^u}, respectively.
Under Assumption \ref{ass:markovian}, that is, premiums have a Markovian structure,   we get  for all $t \in [0,T]$
\begin{equation}
\label{nuova}
V_t = \varphi(t, \lambda_t),
\end{equation}
where $\varphi(t, \lambda)$ is given in \eqref{varphi}.
Our aim is to develop a BSDE characterization for the process $W^N$ which also provides a complete description of the value process $V$ in \eqref{eqn:V}.
The following sets of stochastic processes will play a key role for our BSDE characterization and its solution.
\begin{definition}
We define the following classes of stochastic processes:
\begin{itemize}
\item $\mathcal{L}^2$ denotes the space of c\`adl\`ag $\mathbb{F}$-adapted processes $Y=\{Y_t, \ t \in [0,T]\}$ such that:
\[
\esp{\int_0^T |Y_t|^2 \ud t} <+\infty.
\]
\item $\widehat{\mathcal{L}}^\a$ ($\widehat{\mathcal{L}}^\b$) denotes the space of $[0,+\infty)$-indexed $\mathbb{F}$-predictable random fields $\Theta=\{ \Theta_t(z),\ t \in [0,T],\ z \in [0, + \infty)\}$ such that:
\begin{align}
&\esp{\int_0^T\int_0^{+\infty} \Theta_t^2(z)  \lambda_{t^-} F^\a(\ud z)  \,\ud t } <+\infty,\\
& \left(\mathbb{E} \left[ \int_0^T\int_0^{+\infty} \Theta_t^2(z)  \rho F^\b(\ud z)  \,\ud t \right] <+\infty,\quad \text{respectively} \right).
\end{align}
\end{itemize}
\end{definition}
 

\noindent We now introduce the two-dimensional step process $Z=(C^\a, C^\b)$ where
\begin{equation}
C^\a_t = C_t =\sum_{n=1}^{N^\a_t} Z^\a_n, \quad C^\b_t=\sum_{n=1}^{N^\b_t} Z^\b_n. 
\end{equation}
Let $m(\ud t, \ud z_1, \ud z_2)$ be the integer-valued measure associated to $Z=(C^\a, C^\b)$. 
Recalling that $C^\a$ and $C^\b$ have not common jump times, the following equality holds
\begin{equation}\label{m}
m(\ud t, \ud z_1, \ud z_2) = m^{\a}(\ud t, \ud z_1) \delta_{0}(\ud z_2) + m^{\b}(\ud t, \ud z_2) \delta_{0}(\ud z_1). \end{equation}
Moreover, by \citet[Theorem 4.1]{bandini2024compensator} the $\bF$-dual predictable projection of $m(\ud t, \ud z_1, \ud z_2)$ is given by
\begin{align}
\nu(\ud t, \ud z_1, \ud z_2) = & \nu^{\a}(\ud t, \ud z_1) \delta_{0}(\ud z_2) + \nu^{\b}(\ud t, \ud z_2) \delta_{0}(\ud z_1) \\ 
= & \lambda_{t^-}F^\a(\ud z_1) \delta_{0}(\ud z_2) \ud t+ \rho F^\b(\ud z_2) \delta_{0}(\ud z_1) \ud t, \label{nu}
\end{align}
and let $\tilde{m}(\ud t, \ud z_1, \ud z_2) = m(\ud t, \ud z_1, \ud z_2) - \nu(\ud t, \ud z_1, \ud z_2)$ be the compensated random measure.

\noindent Let us denote by $\widehat{\mathcal{L}}$ the space of $\mathbb{F}$-predictable random fields $\Theta=\{ \Theta_t(z_1,z_2),\ t \in [0,T],\ z_i \in [0, + \infty),\ i=1,2\}$ 
such that:
\[
\esp{\int_0^T\int_0^{+\infty}  \int_0^{+\infty}\Theta_t^2(z_1,z_2)  \nu(\ud t, \ud z_1, \ud z_2) } <+\infty.
\]
Clearly,  $\Theta=\{ \Theta_t(z_1,z_2),\ t \in [0,T],\ z_i \in (0, + \infty),\ i=1,2\} \in \widehat{\mathcal{L}}$ if and only if  $\{\Theta_t(z_1,0),\ t \in [0,T],\ z_1 \in [0, + \infty)\}  \in  \widehat{\mathcal{L}}^\a$ and  $\{\Theta_t(0,z_2),\ t \in [0,T],\ z_2 \in [0, + \infty)\}  \in  \widehat{\mathcal{L}}^\b$.\\

\noindent We are now in the position to provide an $\bF$-martingale representation theorem. 
\begin{proposition}\label{mg representation}
Any square-integrable $(\bF, \Pp)$-martingale $M=\{M_t, t \in [0,T]\} $ has the following representation
\begin{equation}
  M_t = M_0 + \int_0^t \int_{0}^{+\infty} \Gamma^\a_s (z) \widetilde m^\a(\ud s, \ud z) + 
  \int_0^t \int_{0}^{+\infty} \Gamma^\b_s (z) \widetilde m^\b(\ud s, \ud z),
\end{equation}
where $\Gamma^{(i)} \in \widehat{\mathcal{L}}^{(i)}$, $i=1,2$, $\widetilde m^\a(\ud t, \ud z)$ and $\widetilde m^\b(\ud t, \ud z)$ are the compensated integer-valued random measures defined as
\begin{equation}\label{mcomp}
\widetilde m^\a(\ud t, \ud z) =  m^\a(\ud t, \ud z) - \lambda_{t^-} F^\a(\ud z)  \,\ud t, \quad 
\widetilde m^\b(\ud t, \ud z) = m^\b(\ud t, \ud z) - \rho F^\b(\ud z)  \,\ud t. 
\end{equation}
\end{proposition}
  
\begin{proof}
    Since $\bF^{m} = \bF = \bF^{m^\a} \vee \bF^{m^\b}$
    we can apply \citet[Theorem 3.1]{bandini2024compensator}. Hence, any  $(\bF, \Pp)$-local martingale $M$ can be represented as 
    \begin{equation}
  M_t = M_0 + \int_0^t \int_{0}^{+\infty} \int_{0}^{+\infty}\Gamma_s (z_1, z_2) \widetilde m(\ud s, \ud z_1, \ud z_2), \quad t \in [0, T],
\end{equation}
with $\Gamma =\{ \Gamma_t(z_1,z_2),\ t \in [0,T],\ z_i \in [0, + \infty),\ i=1,2\}$  an $\bF$-predictable random field
such that:
\[
\int_0^T\int_0^{+\infty}  \int_0^{+\infty}|\Gamma_t(z_1,z_2)|  \nu(\ud t, \ud z_1, \ud z_2)  <+\infty, \quad \Pp-\mbox{a.s.}.
\]
If in addition $M$ is square integrable, $\Gamma \in \widehat{\mathcal{L}}$ because for any $t \in [0,T]$
$$\esp{M^2_t}=  \esp{\int_0^t\int_0^{+\infty}  \int_0^{+\infty}|\Gamma_t(z_1,z_2)|^2  \nu(\ud t, \ud z_1, \ud z_2) }.$$

\noindent Finally, the statement follows taking the structure of $\widetilde m(\ud t, \ud z_1, \ud z_2)$ into account, see \eqref{m} and \eqref{nu}.
\end{proof}

\noindent We give below a general verification result.  
\begin{proposition} \label{VT1}
Suppose there exists an $\mathbb{F}$-adapted process $D=\{D_t,\ t \in [0,T]\}$ such that:
\begin{itemize}
\item $\{ D_t e^{ - \eta \bar{X}^u_t e^{rT}},\ t \in [0, T] \}$ is an $(\bF,\Pp)$-sub-martingale for any $u \in \mathcal{U}$ and an $(\bF,\Pp)$-martingale for some $u^* \in \mathcal{U}$;
\item $D_T = 1$.
\end{itemize}
Then, $D_t = V_t$ $\Pp$-a.s. for each $t \in [0,T]$ and $u^* $ is an optimal control. Moreover, for any $u \in \mathcal{U}$, $W^u$ is an $(\bF,\Pp)$-sub-martingale and $W^{u^*}$ is an $(\bF,\Pp)$-martingale.
\end{proposition}

\begin{proof}
In view of the terminal condition and the sub-martingale property, for every $t \in [0,T]$ we have
\begin{equation}
    \condespf{e^{ - \eta \bar{X}^u_T e^{rT}}} \ge D_t e^{ - \eta \bar{X}^u_t e^{rT}},\quad {\rm for\ each}\ u \in \mathcal U,
\end{equation}
so that
\begin{equation}
    D_t \leq  \condespf{e^{ - \eta e^{rT}(\bar{X}^u_T-\bar{X}^u_t) }},
\end{equation}
which implies $D_t \leq V_t$ $\Pp$-a.s. for each $t \in [0,T]$. On the other hand, for $u^* \in \mathcal U$, we have that
\begin{equation}
    D_t=\condespf{e^{ - \eta e^{rT}(\bar{X}^{u^*}_T-\bar{X}^{u^*}_t) }} \ge V_t,
\end{equation}
hence $D_t = V_t$ $\Pp$-a.s. for each $t \in [0,T]$. Finally, the last statement, known as the optimal Bellman principle, follows from \eqref{eqn:W^u_and_V}.
\end{proof}
\noindent The first main result is as follows.

\begin{theorem}\label{T1}
Let Assumption \ref{ass_app_premium} be in force. Let $(Y,\Theta^{Y,\a}, \Theta^{Y,\b}) \in \mathcal{L}^2 \times \widehat{\mathcal{L}}^\a \times \widehat{\mathcal{L}}^\b$ be a solution to the BSDE 
\begin{equation} \label{bsde}
\begin{split}
Y_t & = \xi - \int_t^T \int_0^{+ \infty} \Theta^{Y,\a}_s (z) \widetilde m^\a(\ud s, \ud z) - \int_t^T \int_0^{+ \infty} \Theta^{Y,\b}_s (z) \widetilde m^\b(\ud s, \ud z)\\
	& \qquad - \int_t^T \esssup_{u \in \mathcal U}\widetilde f ( s, Y_{s^-} , \Theta^{Y, \a}_s(\cdot),u_s)\,\ud s ,
 \end{split}
\end{equation}
with terminal condition $\xi = e^{- \eta X^N_T}$, where
\begin{multline}\label{ftilde}
\widetilde{f} ( t, Y_{t^-} , \Theta^{Y, \a}_t(\cdot ) , u_t ) = -Y_{t-}\eta e^{r(T-t)} q^{u}_t \\
	+ \int_0^{+\infty} [Y_{t-}+\Theta^{Y, \a}_t (z) ] \big[ 1- e^{-\eta e^{r(T-t)}(z-\Phi(z,u_t))}\big]
	\lambda_{t^-} F^{(1)} (\ud z).
\end{multline}

\noindent Then, $Y_t = W^N_t$ $\Pp$-a.s. and the process $u^* \in \mathcal{U}$ which satisfies
\begin{equation}\label{eqn:u*esssup}
\widetilde{f} ( t, Y_{t^-} , \Theta^{Y, \a}_t(\cdot ) , u^*_t )
	= \esssup_{u \in \mathcal{U}}  \widetilde{f} ( t, Y_{t^-} , \Theta^{Y, \a}_t(\cdot ) , u_t )
	\qquad \forall t\in[0,T]
\end{equation}
provides an optimal control.
\end{theorem}

\begin{proof} First, let us observe that there exists $u^* \in \mathcal{U}$ which satisfies Equation \eqref{eqn:u*esssup}: by hypothesis $q^u_t(\lambda,u)$ and
	$\Phi(z,u)$ are continuous on $u \in U$ and $U$ is compact, hence measurability selection results (see e.g. \citet{benevs1971existence}) ensure that the maximizer is an $(\bF,\Pp)$-predictable process and Proposition \ref{ADM}(b) holds.
 
 Let $(Y,\Theta^{Y, \a}, \Theta^{Y, \b}) \in \mathcal{L}^2 \times \widehat{\mathcal{L}}^\a \times \widehat{\mathcal{L}}^\b$ be a solution to the BSDE \eqref{bsde}
and  $u^*\in {\mathcal U}$ be the process satisfying Equation \eqref{eqn:u*esssup}.
Define $D_t :=  e^{\eta \bar{X}^N_t e^{rT}} \ Y_t$, $t \in [0,T]$,
and observe that $D_T = e^{\eta X^N_T} \xi =1.$ We will prove that 
\begin{itemize}
\item[i)] $\{ D_t e^{ - \eta \bar{X}^u_t e^{rT}} , t \in [0, T] \}$ is an $(\bF,\Pp)$-sub-martingale for any $u \in \mathcal{U}$;  
\item[ii)] $\{ D_t e^{ - \eta \bar{X}^{u^*}_t e^{rT}} , t \in [0, T] \}$ is an $(\bF,\Pp)$-martingale,
\end{itemize}
\noindent then the statement will follow by Proposition \ref{VT1}.

i) By It\^o's product rule, for any $u \in \mathcal{U}$
\begin{align*}
&\ud (D_t \ e^{ - \eta \bar{X}^u_t e^{rT}})  =  \ud  ( e^{  \eta (\bar{X}^N_t - \bar{X}^u_t)e^{rT}} \ Y_t) \\
& =  e^{\eta (\bar{X}^N_{t^-}-\bar{X}^u_{t^-}) e^{rT}}\ \ud Y_t + Y_{t-} \ \ud ( e^{  \eta (\bar{X}^N_t - \bar{X}^u_t)e^{rT}} ) + \ud \left ( \sum_{s \leq t} \Delta Y_s \ \Delta \big(e^{  \eta (\bar{X}^N_s - \bar{X}^u_s)e^{rT} }\big) \right).
\end{align*}
Recalling \eqref{eqn:X_disc}, we notice that
\begin{equation}\label{c1}
\bar{X}^N_{t}-\bar{X}^u_{t} = \int_0^t e^{-rs}q^u_s\,\ud s
-\int_0^t\int_0^{+\infty} e^{-rs} (z-\Phi(z, u_s)) \,m^\a(\ud s,\ud z),
\end{equation}
and applying It\^o's formula we obtain
\begin{align*} \ud ( e^{  \eta (\bar{X}^N_t - \bar{X}^u_t)e^{rT}} ) =  &
\ \eta e^{r T}  e^{  \eta (\bar{X}^N_t - \bar{X}^u_t)e^{rT}} e^{-rt} \ q^u_t \ \ud t \\
&+ e^{  \eta (\bar{X}^N_{t-} - \bar{X}^u_{t-})e^{rT}}  \int_0^{+\infty} \big(e^{-\eta e^{r(T-t)}(z-\Phi(z,u_t))}-1\big) m^\a(\ud t, \ud z).
\end{align*}
Finally, in view of
\[
\ud Y_t = \int_0^{+ \infty} \Theta^{Y,\a}_t (z) \widetilde m^\a(\ud t, \ud z) + \int_0^{+ \infty} \Theta^{Y,\b}_t (z) \widetilde m^\b(\ud t, \ud z) + \esssup_{w \in \mathcal U}\widetilde f ( t, Y_{t^-} , \Theta^{Y, \a}_t(\cdot),w_t ) \ud t,
\]
after some calculations we get, for any $u \in \mathcal{U}$
\begin{multline}\label{Ver1}
\ud (D_t e^{ - \eta \bar{X}^u_t e^{rT}}) = \ud M^u_t +  e^{ \eta (\bar{X}^N_{t} - \bar{X}^u_{t})e^{rT}}
\biggl( \esssup_{w \in \mathcal{U}}  \widetilde{f} ( t, Y_{t^-} , \Theta^{\a}_t(\cdot ) , w_t ) -  \widetilde{f} ( t, Y_{t^-} ,\Theta^{\a}_t(\cdot ) , u_t ) \biggr) \ud t ,
\end{multline}
where for all $t\in[0,T]$
\begin{align}
    M^u_t = &\int_0^t \int_0^{+\infty} e^{\eta (\bar{X}^N_{s^-}-\bar{X}^u_{s^-}) e^{rT}} \ \Theta^{Y, \a}_s (z) \ e^{-\eta e^{r(T-s)}(z-\Phi(z,u_s))} \ \widetilde m^\a(\ud s, \ud z) \\
& +\int_0^t \int_0^{+\infty} Y_{s-} \ e^{\eta (\bar{X}^N_{s^-}-\bar{X}^u_{s^-}) e^{rT}} \biggl(e^{-\eta e^{r(T-s)}(z-\Phi(z,u_s))}-1\biggr) \ \widetilde m^\a(\ud s, \ud z) \\
& +\int_0^t \int_0^{+\infty} e^{\eta (\bar{X}^N_{s^-}-\bar{X}^u_{s^-}) e^{rT}}  \Theta^{Y, \b}_s (z) \ \widetilde m^\b(\ud s, \ud z). 
\end{align}
It remains to verify that, for any $u\in\mathcal{U}$, the process $\{M^u_t,\ t\in[0,T]\}$,
is an $(\bF,\Pp)$-martingale. To this end, it is sufficient to prove that the following three integrability conditions hold, 
\begin{gather*}
\mathbb{E} \left[ \int_0^T \int_0^{+\infty}  e^{\eta (\bar{X}^N_{t}-\bar{X}^u_{t}) e^{rT}} \ \big|\Theta^{Y, \a}_t (z) \big| e^{-\eta e^{r(T-t)}(z-\Phi(z,u_t))} \lambda_t F^\a(\ud z) \ud t \right] < +\infty, \\
\mathbb{E} \left[ \int_0^T \int_0^{+\infty}  \ e^{\eta (\bar{X}^N_{t}-\bar{X}^u_{t}) e^{rT}} \ |Y_t| \big| e^{-\eta e^{r(T-t)}(z-\Phi(z,u_t))}-1 \big| \lambda_t F^\a(\ud z) \ud t \right] < +\infty, \\
\mathbb{E} \left[ \int_0^T \int_0^{+\infty}  e^{\eta (\bar{X}^N_{t}-\bar{X}^u_{t}) e^{rT}} \ \big|\Theta^{Y, \b}_t (z) \big|  \rho F^\b(\ud z) \ud t \right] < +\infty.
\end{gather*}
Using \eqref{c1}, $\Phi(z,u_t) \leq z$, the well known inequality $2ab \leq a^2 + b^2$, $\forall a,b \in \R$ with the choice $a= e^{  \eta e^{rT} \int_0^T e^{-rt} q^{u_M}_t\,\ud t  }$ and $b=\Theta^{Y, \a}_t (z)$ the first expectation above is dominated by
\[
\begin{split}
&\mathbb{E} \left[ e^{  \eta e^{rT} \int_0^T e^{-rt} q^{u_M}_t\,dt  } \int_0^T  \int_0^{+\infty} \bigl| \Theta^{Y, \a}_t (z) \bigr| \lambda_t F^\a(\ud z) \ud t \right]  \\
&\leq \frac{1}{2} \left \{ \mathbb{E} \left[ e^{  2\eta e^{rT} \int_0^T e^{-rt}q^{u_M}_t\,\ud t  } \int_0^T \lambda_t \ud t \right] + \mathbb{E} \left[ \int_0^T \int_0^{+\infty} \bigl| \Theta^{Y, \a}_t (z) \bigr|^2  \lambda_t F^\a(\ud z) \ud t \right] \right \}. \end{split}
\]
Now, applying again inequality $2ab \leq a^2 + b^2$, for all $a,b \in \R$, with the choice  $a= e^{  \eta e^{rT} \int_0^T e^{-rt} q^{u_M}_t\,\ud t}$ and $b=\lambda_t$  we obtain
\[
\begin{split}
&\mathbb{E} \left[ \int_0^T \int_0^{+\infty}  e^{\eta (\bar{X}^N_{t}-\bar{X}^u_{t}) e^{rT}} \ \big|\Theta^{Y, \a}_t (z) \big| e^{-\eta e^{r(T-t)}(z-\Phi(z,u_t))} \lambda_t F^\a(\ud z) \ud t \right] \\
&\leq \frac{1}{4} \mathbb{E} \left[ e^{  4\eta e^{rT} \int_0^T e^{-rt}q^{u_M}_t\,\ud t  } \right] T
	+ \frac{1}{4} \mathbb{E} \left[ \int_0^T \lambda^2_t \ud t\right]\\
 & \quad \quad 
	+ \frac{1}{2} \mathbb{E} \left[ \int_0^T \int_0^{+\infty} \bigl| \Theta^{Y, \a}_t (z) \bigr|^2  \lambda_t F^\a(\ud z) \ud t \right] < +\infty,
\end{split}
\]
which is finite in view of Assumption \ref{ass_app_premium} ii), Proposition \ref{mom} and recalling that $\Theta^{Y, \a} \in \widehat{\mathcal{L}}^\a$. Similarly, the second expectation is lower than
\[
\begin{split}
	&\mathbb{E} \left[ e^{  \eta e^{rT} \int_0^T e^{-rt}q^{u_M}_t\,\ud t  }
	\int_0^T |Y_{t}| \ \lambda_t \ud t \right] \\
	&\le \frac{1}{2} \mathbb{E} \left[ \int_0^T \abs{Y_{t}}^2  \ud t \right]
	+ \frac{1}{4} \mathbb{E} \left[ e^{  4\eta e^{rT} \int_0^T e^{-rt}q^{u_M}_t\,\ud t  } \right] T
	+ \frac{1}{4} \mathbb{E} \left[ \int_0^T \lambda^4_t \ud t \right] <+\infty ,
\end{split}
\]
where the first term is finite because $Y \in \mathcal{L}^2$, the second is finite by Assumption  \ref{ass_app_premium} $ii)$ and the third follows by Proposition \ref{mom}.
Finally, the third expectation is lower than 
\[
\begin{split}
&\mathbb{E} \left[ \int_0^T \int_0^{+\infty}  e^{\eta (\bar{X}^N_{t}-\bar{X}^u_{t}) e^{rT}} \ \big|\Theta^{Y, \b}_t (z) \big|  \rho F^\b(\ud z) \ud t \right] \\
&\leq \frac{1}{2} \left \{ \mathbb{E} \left[ e^{  2\eta e^{rT} \int_0^T e^{-rt}q^{u_M}_t\,\ud t  } \right] + \mathbb{E} \left[ \int_0^T \int_0^{+\infty} \bigl| \Theta^{Y, \b}_t (z) \bigr|^2  \rho F^\b(\ud z) \ud t \right] \right \} <  +\infty.\end{split}
\]
  ii) Choosing $u=u^*$ in \eqref{Ver1} we get 
 $$ \ud (D_t e^{ - \eta \bar{X}^{u^*}_t e^{rT}}) = \ud M^{u^*}_t,$$
that is, $\{ D_t e^{ - \eta \bar{X}^{u^*}_t e^{rT}},\  t \in [0, T] \}$ is an $(\bF,\Pp)$-martingale.
\end{proof}

\noindent We discuss existence of solution to BSDE \eqref{bsde}. Due to unboundedness of the claim arrival intensity, the generator of the BSDE is not Lipschitz and then we rely on  \citet{Papa_Possa_Sapla2018}, where existence results for BSDEs with stochastic Lipschitz generators are provided.  

\begin{theorem}\label{T2}
Let Assumption \ref{ass_app_premium} be in force. 
There exists a unique solution $(Y,\Theta^{Y,\a}, \Theta^{Y,\b}) \in \mathcal{L}^2 \times \widehat{\mathcal{L}}^\a \times \widehat{\mathcal{L}}^\b$ to the BSDE \eqref{bsde}.
\end{theorem}
\noindent The proof is postponed to Appendix \ref{appendix:proofs}.
Gathering the results given in Theorem \ref{T1} and Theorem \ref{T2} yields the following.
\begin{corollary}\label{corollary}
Let Assumption \ref{ass_app_premium} be in force. Then,    
  \begin{itemize}
    \item[i)] $(W^N, \Theta^\a, \Theta^\b)  \in \mathcal{L}^2 \times \widehat{\mathcal{L}}^\a \times \widehat{\mathcal{L}}^\b$ is the unique solution to BSDE 
    \begin{equation} \label{bsde2}
\begin{split}
W^N_t & = \xi - \int_t^T \int_0^{+ \infty} \Theta^{\a}_s (z) \widetilde m^\a(\ud s, \ud z) - \int_t^T \int_0^{+ \infty} \Theta^{\b}_s (z) \widetilde m^\b(\ud s, \ud z)\\
	& \qquad - \int_t^T \esssup_{u \in \mathcal U}\widetilde f ( s, W^N_{s^-} , \Theta^{\a}_s(\cdot),u_s)\,\ud s ,
 \end{split}
\end{equation}
with terminal condition $\xi = e^{- \eta X^N_T}$, where
\begin{multline}\label{ftilde1}
\widetilde{f} ( t, W^N_{t^-} , \Theta^{\a}_t(\cdot ) , u_t ) = - W^N_{t-}\eta e^{r(T-t)} q^{u}_t \\
	+ \int_0^{+\infty} [W^N_{t-}+\Theta^{\a}_t (z) ] \big[ 1- e^{-\eta e^{r(T-t)}(z-\Phi(z,u_t))}\big]
	\lambda_{t^-} F^{(1)} (\ud z).
\end{multline} 
This completely characterizes the value process $\{V_t = e^{\eta \bar{X}^N_t e^{rT}} W^N_t,\ t \in [0,T]\}$ associated to the optimal reinsurance problem \eqref{eqn:minpb}; 
\item[ii)] any process $u^* \in \mathcal{U}$ which maximizes $\widetilde{f} ( t, W^N_{t^-} , \Theta^{\a}_t(\cdot ) , u_t )$ furnishes an optimal reinsurance strategy.
\end{itemize} 
\end{corollary}

\begin{remark}\label{JUMPS}
The process $W^N$ completely determines the predictable random fields $\Theta^\a$ and $\Theta^\b$ outside a null set with respect to the measure $F^\a(\ud z) \ud t$ and $F^\b(\ud z) \ud t$. Indeed, let $\widetilde \Theta^{(i)}_t(z)$ an  $\bF$-predictable random field such that, for any $i=1,2$
$$\Delta W^N_{T_n^{(i)}} = \Theta^{(i)}_{T_n^{(i)}}(Z_n^{(i)})=\widetilde \Theta^{(i)}_{T_n^{(i)}}(Z_n^{(i)}).$$
Thus, for any $t \in [0,T]$, Borel set $C$ of $[0,+ \infty)$ and $i=1,2$
$$
0= \mathbb{E} \left[ \int_0^t \int_C \left | \Theta^{(i)}_t(z) - \widetilde \Theta^{(i)}_t(z) \right | m^{(i)}( \ud s, \ud z)\right ] = \mathbb{E} \left[ \int_0^t \int_C \left | \Theta^{(i)}_t(z) - \widetilde \Theta^{(i)}_t(z) \right | \nu^{(i)}( \ud s, \ud z) \right ]
$$
and, recalling equation \eqref{dualpred}, it implies that $\Theta^{(i)}_t(z)(\omega)= \widetilde \Theta^{(i)}_t(z)(\omega)$, $F^{(i)}(\ud z) \ud t P(\ud \omega)$-a.e., for $i=1,2$.
\end{remark}
In the sequel, we will focus on a Markovian framework, that is, Assumption \ref{ass:markovian} holds and the next result provides a more explicit representation for $\Theta^{(1)}$  and $\Theta^{(2)}$ in terms of the value function $\varphi(t, \lambda)$ given in \eqref{varphi}.
\begin{proposition}\label{phi}
Under Assumption \ref{ass:markovian} the following equalities hold
\begin{align} 
 \Theta^{(1)}_t(z)  & = e^{-\eta \bar{X}^N_{t-} e^{rT}} \left[ e^{\eta z e^{r(T-t)}} \varphi(t, \lambda_{t-} +\ell(z)) -  \varphi(t, \lambda_{t-}) \right ], \quad  F^{(1)}(\ud z) \ud t \ud \Pp-a.e.\label{jump1}\\
\Theta^{(2)}_t(z)  & = e^{-\eta \bar{X}^N_{t-} e^{rT}} \left[ \varphi(t, \lambda_{t-} +z) -  \varphi(t, \lambda_{t-}) \right] \quad  F^{(2)}(\ud z) \ud t \ud \Pp-a.e. \label{jump2}
\end{align}
\end{proposition}
\begin{proof}
Under Assumption \ref{ass:markovian} for any $t \in [0,T]$, $V_t = \varphi(t, \lambda_t)$.
From \eqref{eqn:VisWI} and recalling that 
$\bar X^N_t = R_0 + \int_0^t e^{-rs} c_s \,\ud s
-\int_0^t\int_0^{+\infty} e^{-rs} z\,m^\a(\ud s,\ud z)$ 
we get that $\bar{X}^N_{T_n^\a} = \bar{X}^N_{T_n^{\a^-}} - Z^\a_n e^{- r T_n^\a}$. Thus,
\begin{equation}\begin{split}
    \Theta^{(1)}_{T_n^{(1)}}(Z_n^{(1)}) = & W^N_{T_n^\a} - W^N_{T_n^{\a^-}} = \exp \left ({-\eta \bar{X}^N_{T_n^\a}{e^{rT}} }\right ) V_{T_n^\a} - \exp\left ({-\eta \bar{X}^N_{T_n^{\a^-}}{e^{rT} }}\right ) V_{T_n^{\a^-}} \\
    = & \exp\left ({-\eta \bar{X}^N_{T_n^{\a^-}}{e^{rT} }}\right ) \big ( V_{T_n^\a}e^{\eta Z^\a_n e^{r(T-T_n^\a)} } - V_{T_n^{\a^-}}\big ) \\
    =& \exp\left ({-\eta \bar{X}^N_{T_n^{\a^-}}{e^{rT} }}\right ) \big ( \varphi({T_n^\a}, \lambda_{T_n^\a} ) e^{\eta Z^\a_n e^{r(T-T_n^\a)} } - \varphi({T_n^{\a^-}}, \lambda_{T_n^{\a^-}} )\big ). 
\end{split}
\end{equation}
 \noindent Now taking \eqref{intensity_eq} into account,  we have that $\lambda_{T_n^\a} = \lambda_{T_n^{\a^-}}  + \ell(Z^\a_n)$ and from   Remark \ref{JUMPS} we obtain \eqref{jump1}. Analogously, since  
$\Delta \bar{X}^N_{T_n^\b} =0 $ and 
$\lambda_{T_n^\b} = \lambda_{T_n^{\b^-}}  + Z^\b_n$, we can write
\begin{equation}\begin{split}
    \Theta^{(2)}_{T_n^{(2)}}(Z_n^{(2)}) = & W^N_{T_n^\b} - W^N_{T_n^{\b^-}} = e^{-\eta \bar{X}^N_{T_n^\b}{e^{rT}} } \big ( V_{T_n^\b} -  V_{T_n^{\b^-}} \big )\\
    =& e^{-\eta \bar{X}^N_{T_n^\b}{e^{rT}} } \big ( \varphi({T_n^\b}, \lambda_{T_n^{\b^-}}  + Z_n^\b)  - \varphi({T_n^{\b^-}}, \lambda_{T_n^{\b^-}} )\big ), 
    \end{split}\end{equation}
    which implies \eqref{jump2}. 
\end{proof}


\section{The optimal reinsurance strategy}\label{Optimal_Reinsurance}

\noindent The purpose of this section is to provide more insight into the structure of the optimal reinsurance strategy and investigate some special cases. We focus on a Markovian setting, and so we make Assumption \ref{ass:markovian} in force.\\
The following general result provides a characterization of the optimal reinsurance strategy in the one-dimensional case, where $\Phi(z,u)$ is increasing in $u$, with $u \in [u_M, u_N] \subset \overline{ \mathbb R}$.
In order to obtain some definite results we need to introduce a concavity hypothesis for the function $\widetilde{f}$ given in \eqref{ftilde1} with respect to the variable $u \in [u_M,u_N]$.

\begin{proposition}\label{prop:optreins_general}
Under Assumption \ref{ass_app_premium}, 
suppose that $\Phi (z, u)$ is differentiable in $u\in [u_M,u_N]$ for almost every $z \in (0, + \infty)$ and $\widetilde{f}$ given in Equation \eqref{ftilde1} is strictly concave in $u\in [u_M,u_N]$. Then, the optimal reinsurance strategy $u^*_t=\{ u^*(t,\lambda_{t^-}), t\in[0,T] \}$ is given by \eqref{u*_general_HJB}.
\end{proposition}

\begin{proof}
To stress the dependence of $\widetilde f$, given in \eqref{ftilde1}, on $\lambda$ and to explicitly feature $\varphi(t,\lambda_{t^-})$ into its expression, we introduce the function $\bar f(t,\lambda,u)$ by setting
$\bar{f}(t, \omega, \lambda_{t^-}(\omega), u):= \widetilde{f} ( t, \omega, W^N_t(\omega), \Theta^\a_t ( \cdot )(\omega), u)$,
which is continuous and strictly concave in $u\in[u_M,u_N]$ by hypothesis.  Indeed, from  \eqref{eqn:VisWI}, \eqref{nuova} and \eqref{jump1}, $\bar{f}(t, \lambda_{t^-}, u)$ can be written as
\begin{equation}
\begin{split}
\bar{f}(t, \lambda_{t^-}, u)= & -e^{- \eta \bar X^N_{t-} e^{rT}} \Big ( \varphi(t, \lambda_{t^-} )\eta e^{r(T-t)} q(t,\lambda_{t^-},u_t)  \\
	& - \int_0^{+\infty} \varphi(t, \lambda_{t^-} + \ell(z)) e^{\eta e^{r(T-t)z}} \big[ 1- e^{-\eta e^{r(T-t)}(z-\Phi(z,u_t))}\big]
	\lambda_{t^-} F^{(1)} (\ud z) \Big ).
 \end{split}
\end{equation}
Thus, the first order condition
\begin{equation}
    \label{nn}
 \frac{\partial \bar{f}}{\partial u}(t,  \lambda, u) = \eta e^{r(T-t)} e^{- \eta \bar X^N_{t-} e^{rT}} h(t, \lambda,u)=0,
 \end{equation}
reads as $h(t, \lambda,u)=0$, with $h(t, \lambda,u)$ defined in \eqref{eq:h}, and we obtain \eqref{1st_order_genHJB}. 
Finally, recalling that $ \frac{\partial \bar{f}}{\partial u}(t,  \lambda, u) $ is a decreasing function on $u\in [u_M, u_N]$ we 
have three possible cases, in terms of the regions $A_0$ and $A_1$ defined in Proposition \ref{strategyHJB}:

(i) for $(t,\lambda) \in A_0$,  $\ds \frac {\partial \bar{f}}{\partial u}(t,\lambda,u) <  0 $, so $\bar{f}(t, \lambda,u)$ is decreasing in $u\in  [u_M, u_N]$, hence $u^*(t, \lambda) = u_M$;

(ii) for $(t,\lambda) \in A_1$,  $\ds\frac {\partial \bar{f}}{\partial u}(t,\lambda,u) > 0 $, so $\bar{f}(t, \lambda,u)$ is increasing in $u\in [u_M, u_N]$, hence $u^*(t, \lambda) = u_N$;

(iii) otherwise there exists
$\bar{u}(t,\lambda) \in (u_M,u_N)$ such that $\ds \frac{\partial \bar{f}}{\partial u}(t,\lambda,u)\Big{|}_{u=\bar u}=0$ and
$u^*(t, \lambda) = \bar u(t, \lambda)$. 
 
\end{proof}

\begin{remark}
If $q(t, \lambda,u)$ and $\Phi (z, u)$ are linear or convex on $u \in  [u_M,u_N]$ then $\widetilde{f}$ is strictly concave in $u\in[u_M,u_N]$
and Proposition \ref{prop:optreins_general} applies.
\end{remark}

Let us observe that using the BSDEs approach we obtain a quasi-explicit expression for the optimal reinsurance strategy aligning with Proposition \ref{strategyHJB}, but without any regularity assumption on the value function $\varphi(t,\lambda)$, which is required in the HJB-equation approach.

In the case of Cox process with shot noise intensity, $\ell(z)=0$ for any $z>0$, the function $h(t, \lambda,u)$ reads as
$h(t, \lambda,u) =  
 - \varphi(t,\lambda) \gamma(t, \lambda, u)$
 with 
 \begin{equation}\label{gamma}
 \gamma(t, \lambda, u) = \frac{\partial q}{\partial u}(t, \lambda,u) + \int_{0}^{+\infty}  e^{\eta e^{r (T-t)}\Phi (z, u)}  \frac{\partial \Phi}{\partial u}(z,u) \lambda F^{(1)} (\ud z).
 \end{equation} 
Hence, since $\varphi(t,\lambda)>0$ (see Remark \ref{positive}) the two regions $A_0$, $A_1$ and the optimal control do not depend on the function $\varphi(t,\lambda)$ anymore. Precisely, we have the following result.

\begin{proposition}[Cox process with shot noise intensity]
\label{COX prop}
Under the same assumptions as Proposition \ref{prop:optreins_general}, and with the additional condition that
$\ell(z) =0$ for any $z>0$, 
then, the optimal reinsurance strategy $u^{*,cox}_t=\{ u^{*,cox}(t,\lambda_{t^-}), t\in[0,T] \}$ is given by
\begin{equation}\label{u*_cox}
u^{*,cox}(t,\lambda_{t^-}) =
\begin{cases}
u_M		& (t,\lambda_{t^-}) \in A_0 \\
u_N			& (t,\lambda_{t^-})\in A_1,\\
\bar{u}^{cox}(t,\lambda_{t^-})	& \text{otherwise}, 
\end{cases}
\end{equation}
where 
\begin{align*}
A_0 &= \left\{ (t,\lambda) \in [0,T] \times (0, + \infty) : \gamma(t, \lambda, u_M ) > 0 \right\} \\
 A_1 &=  \left\{ (t,\lambda) \in [0,T] \times (0, + \infty) : \gamma(t, \lambda, u_N ) < 0 \right\},
\end{align*}
 with $\gamma(t, \lambda,u)$ given in \eqref{gamma} and $\bar{u}^{cox}(t,\lambda) \in (u_M,u_N)$ solves the following equation
\begin{equation}\label{1st_order_gen1}
-   \frac{\partial q}{\partial u}(t, \lambda,u) = \int_{0}^{\infty}  e^{\eta e^{r (T-t)}\Phi (z, u)}  \frac{\partial \Phi}{\partial u}(z,u) \lambda F^{(1)} (\ud z).
\end{equation}
\end{proposition}

\begin{remark}\label{cox_comp}
    When $q(t,\lambda, u) = \lambda d(t, u)$, with $d(t,u)$ deterministic function in $(t,u) \in [0,T] \times [u_N, u_M]$, as for the  premium considered in Example \ref{ex_EVP}, we get that $\gamma(t,\lambda,u) = \lambda \bar {\gamma}(t,u)$ with
    $$\bar {\gamma}(t,u) = \frac{\partial d}{\partial u}(t,u) + \int_{0}^{+\infty}  e^{\eta e^{r (T-t)}\Phi (z, u)}  \frac{\partial \Phi}{\partial u}(z,u) F^{(1)} (\ud z)$$ and the optimal strategy turns out to be a deterministic function on time. Precisely,
    \begin{equation}\label{bo}
u^{*,cox}(t) =
\begin{cases}
u_M		& t \in {\bar A}_0 \\
u_N			& t\in {\bar A}_1,\\
\bar{u}^{cox}(t)	& \text{otherwise},
\end{cases}
\end{equation}
where 
$$ 
{\bar A}_0 = \left\{ t\in [0,T] : \bar{\gamma}(t, u_M ) > 0 \right\}, \quad 
 {\bar A}_1 =  \left\{ t\in [0,T] : \bar{\gamma}(t, u_N ) <å 0 \right\}.$$

It is noteworthy that the optimal strategy remains independent of the claim arrival intensity or its dynamics, thus aligning with the optimal strategy under a constant claim arrival intensity. This implies that, within the considered premium principles (see Remark \ref{ex_EVP}), mitigating the risk arising from externally-excited jumps can be achieved solely by adjusting the premium rate. However, addressing the self-exciting effect requires adjustments in both the premium rate and reinsurance strategy, as outlined in Proposition \ref{prop:optreins_general}. In essence, the self-exciting effect poses a greater challenge, necessitating a more comprehensive array of risk management tools. A similar finding was reported by \citet{Cao_Landriault_Li} in a similar contagion model under EVP and the mean–variance criterion.
\end{remark}


\subsection{Optimal reinsurance under EVP}

We discuss proportional reinsurance and limited Excess-of-Loss with fixed reinsurance coverage, see Example \ref{ex_EVP}, when the reinsurance premium is computed according to the Expected Value Principle (EVP).

\subsubsection{Proportional reinsurance}

Let $\Phi(z,u) =z u$, $u \in [0,1]$.
According to \eqref{ex_EVP1}, the reinsurance premium reads as:
\begin{equation}\label{eqn:evp}
q_t^u = (1+ \theta _R ) \mathbb{E} [Z^{(1)}] \lambda_{t^- } (1-u_t), \quad \forall u \in \mathcal U.
\end{equation}
Notice that Assumption \ref{ass_app_premium} $ii)$ is automatically satisfied, since from Proposition \ref{ADM} (a), for every $a >0$,  
$\mathbb E \left[ e^{a \int_0^T \lambda_t \,dt} \right] < +\infty$.

\noindent In this special case we have the following result.

\begin{proposition}\label{propC}
Under Assumption \ref{ass_app_premium} $i)$, there exist two stochastic thresholds $\theta^F(t, \lambda_{t-})< \theta^N(t, \lambda_{t-})$ such that the optimal retention level is given by:
\begin{equation}
\label{ustar}
u^*_t= u^* (t, \lambda_{t-}) =
\begin{cases}
	0  & \text{if } \theta_R < \theta^F(t, \lambda_{t-})  
	\\
	1 & \text{if }  \theta_R >  \theta^N(t, \lambda_{t-}) 
	\\
	\bar{u}(t,\lambda_{t^-})	& \text{otherwise,}
\end{cases}
\end{equation}
where
\begin{align*}
\theta^F(t, \lambda) &= \frac{1}{ \mathbb{E} [Z^{(1)}] }
	\int_0^{\infty} \frac{\varphi(t,\lambda + \ell(z))}{\varphi(t,\lambda)} z F^{(1)} (\ud z) -1, \\
\theta^N(t, \lambda) &= \frac{1}{\mathbb{E} [Z^{(1)}] }
	\int_0^{\infty} \frac{\varphi(t,\lambda + \ell(z))}{\varphi(t,\lambda)} e^{\eta e^{r(T-t)} z} z  F^{(1)} (\ud z) -1
\end{align*}
and where $\bar{u}(t, \lambda) \in (0,1)$ solves the following equation, with respect to $u$:
\begin{equation}
\label{eqn:evp_prop_stationary}
(1+ \theta _R ) \mathbb{E} [Z^{(1)}] = \int_0^{+\infty}  \frac{\varphi(t,\lambda + \ell(z))}{\varphi(t,\lambda)} z e^{\eta e^{r (T-t)} z u}  F^{(1)} (\ud z).
\end{equation}
\end{proposition}
\begin{proof}
It follows immediately from Proposition \ref{prop:optreins_general}.
\end{proof}

Let us briefly comment the previous result. We can distinguish three cases, depending on the stochastic intensity through the value function:
\begin{itemize}
\item  if the reinsurer's safety loading $\theta_R$ is smaller than the $\theta^F(t, \lambda_{t-})$, then  full reinsurance is optimal;
\item if  $\theta_R$ is larger than $\theta^N(t, \lambda_{t-})$, then null reinsurance is optimal and the contract is not subscribed;
\item lastly, if $\theta^F(t, \lambda_{t-}) \leq \theta_R \leq \theta^N(t, \lambda_{t-})$, then the optimal retention level  takes values in $(0,1)$, that is, the ceding company transfers to the reinsurance a non null percentage of risk (not the full risk).
\end{itemize}

 If the value function $\varphi(t,\lambda)$ is a strictly increasing function of $\lambda \in (0, + \infty)$ then $\theta^F(t, \lambda)>0$ for any $(t, \lambda) \in [0,T] \times (0, + \infty)$. As a consequence, full reinsurance may be allowed. In Section \ref{sec:comparison}  we will discuss the monotonicity property of the value function.

\begin{remark}[Cox process with shot noise intensity]\label{RemCox}
In case $\ell(z)=0$ we have that $\theta^F=0$ (i.e. full reinsurance is never optimal) and
$\theta^N(t)  = \frac{1}{\mathbb{E} [Z^{(1)}] }
	\int_0^{\infty} e^{\eta e^{r(T-t)} z} z  F^{(1)} (\ud z) -1>0$ for any $t \in [0,T)$.
 Thus the optimal reinsurance strategy is a deterministic function on time given by
\begin{equation}\label{ustarCox}
u^{*,cox}(t)= 
\begin{cases}
	1 & \text{if }  \theta_R >  \theta^N_t(t) 
	\\
	\bar{u}^{cox}(t)	& \text{if } \theta_R \leq \theta^N_t(t),
\end{cases}
\end{equation}
where $\bar{u}^{cox}(t)\in (0,1)$ is the solution to 
 $(1+ \theta _R ) \mathbb{E} [Z^{(1)}] = \int_0^{+\infty}  z e^{\eta e^{r (T-t)} z u}  F^{(1)} (\ud z)$.
 As already observed in Remark \ref{cox_comp} the same result is obtained when the claim arrival intensity is constant. 
\end{remark}

\subsubsection{Limited Excess-of-Loss with fixed reinsurance coverage}\label{sec:LXL}
 
The reinsurer's loss function is (see Example \ref{ex_reinsurance}$(3)$):
\begin{equation}
z - \Phi (z,u) =  z- \Phi(z,(u_1,u_2)) = (z- u_1 )^{+} - (z- u_2 )^{+} =
\left\{
\begin{array}{lcl}
0 & \textrm{if} & z \le u_1 \\
z - u_1 & \textrm{if} & z \in (u_1,u_2) \\
u_2- u_1 & \textrm{if} & z \ge u_2 , \\
\end{array}
\right.
\end{equation}
 with $u_1 < u_2$, so that the retention function is $\Phi (z,u) = z-  (z- u_1 )^{+} + (z- u_2 )^{+}$, for every $(z,u) \in [0,+\infty) \times U$.
 
To obtain explicit results we will reduce our analysis to the case where the control is $u=u_1$, while $u_2=u_1+\beta_M$ is unequivocally determined, with $\beta_M>0$ being the fixed maximum reinsurance coverage.
According to \eqref{ex_EVP1}, the EVP becomes
\begin{equation}\label{eq:q_stop_loss}
q(\lambda_{t^-},u_t) =  (1+\theta_R)\lambda_{t^-}  \int_{u_t}^{u_t+\beta_M} (1- F^{(1)} (z)) \ud z , \quad \forall u \in \mathcal U.
\end{equation}
Observe that the Assumption \ref{ass_app_premium} $ii)$ is automatically satisfied since by Proposition \ref{ADM} (a), for every $a >0$,  
$\mathbb E \left[ e^{a \int_0^T \lambda_t \,\ud t} \right] < +\infty$.
\begin{proposition}
\label{prop:u*LSL}
Under Assumption \ref{ass_app_premium} $i)$, there exists a stochastic threshold $\theta^L(t, \lambda_{t-})$ such that
\begin{equation}
\label{ustar_LSL}
u^*_t = u^*(t, \lambda_{t-})=
\begin{cases}
	0  & \text{if } \theta_R < \theta^L(t, \lambda_{t-})
	\\
	\bar{u}(t,\lambda_{t-}) & \text{if} \hskip 2mm \theta_R \geq \theta^L(t, \lambda_{t-}),
\end{cases}
\end{equation}
where
\begin{equation}\label{thetaL}
\theta^L(t, \lambda) = \frac{1}{F^\a(\beta_M) }
	\int_0^{\beta_M} \frac{\varphi(t, \lambda + \ell(z)) }{\varphi(t, \lambda)}  F^{(1)} (\ud z) -1,
\end{equation}
and $\bar{u}(t, \lambda) \in (0, +\infty)$ solves the following equation with respect to $u$:
\begin{equation}
\label{eqn:evp_prop_stationaryLSL}
(1+ \theta _R ) \bigl(F^\a(u+\beta_M)-F^\a(u)\bigr) =
	e^{\eta e^{r(T-t)} u} \int_u^{u+\beta_M}  \frac{\varphi(t, \lambda + \ell(z)) }{\varphi(t, \lambda)}  F^{(1)} (\ud z).
\end{equation}
\end{proposition}

\begin{proof}
It is immediate to verify that the assumptions of Proposition \ref{prop:optreins_general} are satisfied,  because the premium in \eqref{eq:q_stop_loss} is convex in $u$ an $\Phi(z, u )$ is differentiable in $u\in (0,+ \infty)$ for any $z \neq u, u + \beta_M$. Moreover, $\frac{\partial \Phi(z, u )}{\partial u }=1$ for $z\in (u,u+\beta_M)$, while it is null elsewhere. 
Notice that the function $h(t, \lambda,u)$ given in \eqref{eq:h} in this special case reads as 
$$
h(t, \lambda,u)= (1+ \theta _R ) \lambda \varphi(t, \lambda) \bigl(F^\a(u+\beta_M)-F^\a(u)\bigr) -
	 \int_u^{u+\beta_M}  \varphi(t, \lambda + \ell(z)) \lambda e^{-\eta e^{r(T-t)}u} F^{(1)} (\ud z).
  $$
Thus, $h(t, \lambda,u_M) = h(t, \lambda,0) = (1+ \theta _R ) \lambda \varphi(t, \lambda) F^\a(\beta_M) -
	 \int_0^{\beta_M}  \varphi(t, \lambda + \ell(z)) \lambda  F^{(1)} (\ud z)$ and
  $h(t, \lambda,u_N)= h(t, \lambda,+ \infty) =0$. 
  As a consequence, null reinsurance is never optimal because $A_1 = \emptyset$ and 
when $\theta_R < \theta^L(t, \lambda_{t-})$, with $\theta^L(t, \lambda_{t-})$ given in \eqref{thetaL}, $u^*_t = u_M = 0$, i.e. the maximal coverage $\beta_M$ is optimal.  On the other hand, if $\theta_R \ge  \theta^L(t, \lambda_{t-})$, $u^*(t,\lambda)$ coincides with  $\bar u(t, \lambda)$ satisfying equation \eqref{1st_order_genHJB}, which corresponds to the solution of equation \eqref{eqn:evp_prop_stationaryLSL}.


\end{proof}
Let us briefly comment the previous result. Differently from the proportional reinsurance, null reinsurance is never optimal and we can distinguish two cases, depending on the maximum coverage $\beta_M$ and the ratio $\ds \frac{\varphi(t, \lambda_{t^-} + \ell(z)) }{\varphi(t, \lambda_{t^-})}$:

\begin{itemize}
\item  if the reinsurer's safety loading $\theta_R$ is smaller than $\theta^L(t, \lambda_{t-})$ then the maximum reinsurance coverage $\beta_M$ is optimal;
\item if  $\theta_R$ is larger than $\theta^L(t,\lambda_{t-})$  then it is optimal purchasing reinsurance but not with maximum coverage.
\end{itemize}

If the value function $\varphi(t,\lambda)$ is a strictly increasing function of $\lambda \in (0, + \infty)$ then $\theta^L(t, \lambda)>0$ for any $(t, \lambda) \in [0,T] \times (0, + \infty)$. As a consequence, maximum reinsurance may be allowed.

\begin{remark}[Cox process with shot noise intensity]\label{cx}
In the case $\ell(z)=0$ we have that $\theta^L=0$ (i.e. maximal reinsurance is never optimal) and equation \eqref{eqn:evp_prop_stationaryLSL} reduces to $(1+ \theta _R ) =
	e^{\eta e^{r(T-t)} u}.$
Thus, for any $\beta_M>0$, the optimal strategy is an increasing function on time given by
 $$
 u^{*,cox}(t) = \frac{\log(1+ \theta_R)}{\eta} e^{-r(T-t)},
 $$
 and coincides with that in the case of constant claim arrival intensity.
\end{remark}

\subsubsection{Excess-of-Loss Reinsurance}

The excess of loss contract, that is, $z - \Phi (z,u) = (z- u )^{+} $ (see Example 4.2(2)) can be easily obtained from the previous case by letting $\beta_M \rightarrow \infty$. The optimal reinsurance strategy, under Assumption \ref{ass_app_premium} $i)$, becomes then:
\begin{equation}
\label{ustar_XL}
u^*_t = u^*(t, \lambda_{t-}) =
\begin{cases}
	0  & \text{if } \theta_R < \theta^L(t, \lambda_{t-})
	\\
	\bar{u}(t,\lambda_{t-})	& \text{if } \theta_R \geq \theta^L(t, \lambda_{t-}),
\end{cases}
\end{equation}
where
\[
\theta^L(t, \lambda) = \int_0^{+\infty} \frac{\varphi(t, \lambda + \ell(z))}{\varphi(t, \lambda) }  F^{(1)} (\ud z) -1
\]
and $\bar{u}(t, \lambda) \in (0, + \infty)$ solves the following equation with respect to $u$:
\begin{equation}
\label{eqn:evp_prop_stationaryXL}
(1+ \theta _R ) (1-F^{(1)} (u) )=
	\int_u^{+\infty}  \frac{\varphi(t, \lambda + \ell(z))}{\varphi(t, \lambda) } e^{\eta e^{r (T-t) u }} F^{(1)} (\ud z).
\end{equation}
As in the Limited Excess-of-Loss with fixed reinsurance coverage case, null reinsurance is never optimal and two cases are possible, depending on 
$\frac{\varphi(t, \lambda_{t^-} + \ell(z))}{\varphi(t, \lambda_{t^-}) }$. 
\begin{itemize}
\item when $\theta_R<\theta^L(t,\lambda_{t-})$, the full reinsurance is optimal;
\item otherwise, it becomes optimal to purchase an intermediate protection level.
\end{itemize}
In the case of Cox process with shot noise intensity, the optimal reinsurance strategy is the same as the Limited Excess-of-Loss reinsurance contract given in Remark \ref{cx}.

\section{Comparison results and monotonicity of the value function}\label{sec:comparison}

\noindent In this section, we assume that reinsurance premiums are computed under the EVP  and compare the optimal strategy in the contagion model with that in the case of Cox process with shot noise intensity corresponding to $\ell(z)=0$ in equation \eqref{intensity}.
First, we focus on proportional reinsurance by giving the following result. 
\begin{proposition} [Proportional Reinsurance] \label{comparison}
 Suppose for any $t \in [0,T]$, $\varphi(t, \lambda)$ is increasing in $\lambda \in (0, + \infty)$. Then, under the EVP and proportional reinsurance, for any $t \in[0,T]$ 
 $$u^*_t = u^{*}(t, \lambda_{t-}) \leq u^{*,cox}(t).$$
 That is, in the contagion model the insurer transfers more risk to the reinsurer than in the case without the self-exciting component. 
 \end{proposition}

 \begin{proof}
First notice that under proportional reinsurance $\Phi(u,z) = uz$ and EVP   from Proposition \ref{eqn:evp_prop_stationary} and Remark \ref {RemCox}, we get that $\bar{u}(t, \lambda)$ and  $\bar{u}^{cox}(t)$ solve 
\begin{align}
(1 + \theta_R) \esp{Z^\a} & = \int_0^{+\infty}  \frac{\varphi(t,\lambda + \ell(z))}{\varphi(t,\lambda)} z e^{\eta e^{r (T-t) z u}} F^{(1)} (\ud z) :=g(t,\lambda, u),\\
(1 + \theta_R) \esp{Z^\a} & =  \int_0^{+\infty}  z e^{\eta e^{r (T-t)} z u}  F^{(1)} (\ud z)   :=g^{cox}(t,u),
\end{align}
respectively.  We fix $t \in [0,T]$ and $\lambda>0$ and we consider the functions  $g(t,\lambda, u)$ and 
$g^{cox}(t, u)$ defined for any $u \in \mathbb R$.
Since $\varphi(t, \lambda)$ is increasing in $\lambda>0$, we have $g(t,\lambda, u) \geq g^{cox}(t, u)$ for any $u \in \mathbb R$. 
Thus, we get that $\bar{u}(t, \lambda) \leq \bar{u}^{cox}(t)$. 
Finally, observing that $u^{*}(t, \lambda) = \max\{ 0, \min\{\bar{u}(t, \lambda), 1\}\}$ and $u^{*, cox}(t) = \max\{ 0, \min\{\bar{u}^{cox}(t), 1\}\}$ we find that for any $(t,\lambda)\in [0,T] \times (0, + \infty) $
$$u^{*}(t, \lambda) \leq u^{*, cox}(t),$$
which implies that $u^*_t = u^{*}(t, \lambda_{t-})\leq u^{*, cox}(t)$, for $t \in [0,T]$. 
 \end{proof}

\noindent We now prove a similar result for limited Excess-of-Loss with fixed maximum reinsurance coverage discussed in Subsection \ref{sec:LXL}. 
\begin{proposition} [Limited excess of loss reinsurance] \label{comparison2}
 Suppose for any $t \in [0,T]$, $\varphi(t, \lambda)$ increasing in $\lambda \in (0, + \infty)$. Then, under the EVP and limited Excess-of-Loss reinsurance, for any $t \in[0,T]$ 
 $$u^*_t = u^{*}(t, \lambda_{t-}) \leq u^{*,cox}(t). 
 $$
 That is, in the contagion model the insurer  transfers more risk to the reinsurer than in the case without the self-exciting component. 
 \end{proposition}

 \begin{proof}
Firstly, recall that under limited Excess-of-Loss  with fixed maximum reinsurance coverage, we have $\Phi(u,z) = z- (z- u )^{+} + (z- u - \beta_M)^{+}$, with $\beta_M > 0$, for every $(z,u) \in [0,+\infty) \times [0,+\infty]$.
Hence, from Proposition \ref{prop:u*LSL}  and Remark \ref{cx} 
we get that $\bar{u}(t, \lambda)$ and  $\bar{u}^{cox}(t)$ solve 
\begin{align}
(1+ \theta _R ) & = 
	\frac{e^{\eta e^{r(T-t)} u}}{F^\a(u+\beta_M)-F^\a(u)} \int_u^{u+\beta_M}  \frac{\varphi(t, \lambda + \ell(z)) }{\varphi(t, \lambda)}  F^{(1)} (\ud z) :=g(t,\lambda, u)\\
(1+ \theta _R )  & =
	\frac{e^{\eta e^{r(T-t)} u}}{F^\a(u+\beta_M)-F^\a(u)} \int_u^{u+\beta_M}  F^{(1)} (\ud z)=e^{\eta e^{r(T-t)} u}:=g^{cox}(t, u),
\end{align}
respectively. Since $\varphi(t, \lambda)$ is increasing in $\lambda>0$, we have $g(t,\lambda, u) \geq g^{cox}(t, u)$, for each $u \in [0,+\infty]$.  This implies that  $\bar{u}(t, \lambda) \leq \bar{u}^{cox}(t)$. Finally, recalling Proposition \ref{prop:u*LSL} and Remark \ref{cx}, we get $u^*_t \leq \bar{u}(t, \lambda_{t-}) \leq \bar{u}^{cox}(t)=u^{*,cox}(t)$, for any $t\in [0,T]$.
 \end{proof}
\noindent In the following, we study the monotonicity property of the value function $\varphi(t,\lambda)$ with respect to $\lambda \in (0, +\infty)$, which is required for the validity of our comparison results. 
Now, we make the assumption.
\begin{assumption}\label{A2_premium}
The insurance and reinsurance premiums are respectively of the form:  $c_t = c(t) \lambda_{t^-}$ and $q^u_t  = \lambda_{t^-} d(t,u_t)$, for each $t \in [0,T]$, where $c(t)>0$ and $d(t,u)$ are deterministic functions.
\end{assumption}

\begin{remark}
Note that, under the classical premiums described in Example \ref{ex_EVP}, Assumption \ref{A2_premium} is satisfied. 
Precisely, under proportional reinsurance, for the EVP, the VPP and the MVP, we have
\begin{align}
    c& =(1+\theta_I) \esp{Z^\a}, \quad d(u)  = (1+\theta_R) \esp{Z^\a}(1-u),\\
    c& = \esp{Z^\a} + \eta_I \esp{(Z^\a)^2},\quad d(u)  = \esp{Z^\a}(1-u) + \eta_R\esp{(Z^\a)^2}(1-u)^2,\\
    c(t)& = \esp{(1+\theta_I(t,Z^\a))Z^\a} + \esp{\eta_I(t,Z^\a)(Z^\a)^2}, \\
     d(t,u) & = \esp{(1+\theta_R(t,Z^\a))Z^\a}(1-u) + \esp{\eta_R(t,Z^\a)(Z^\a)^2}(1-u)^2,
\end{align}
respectively. 
\end{remark}
\noindent We need first a preliminary result.

\begin{lemma}\label {PROP1}
Under Assumption \ref{A2_premium} we have that 
\begin{equation}\label{direct}
\varphi(t, \lambda)
= \inf_{u\in\mathcal{U}_t}\mathbb{E}^\Q_{t,\lambda}\left[e^{\int_t^T \{ \int_0^{+ \infty} (e^{-A(s,u.) z} -1) \rho F^\b(\ud z) + \lambda_s \int_0^{+ \infty} B(s,z,u.) F^\a(\ud z) - a(s,u_s) [ \beta + (\lambda - \beta) e^{-\alpha (s-t)} ] \}\ud s } \right ],
\end{equation}
where $\Q$ is the probability measure equivalent to $\Pp$ defined in \eqref{eqn:L} and $\mathbb{E}^{\Q}_{t,x,\lambda}[\cdot]$ stands for the expectation under $\Q$ when the claim intensity process $\Lambda$ starts from $\lambda$ at time $t$. Moreover, for any $u \in \mathcal{U}$
\begin{align}
a(t,u_t) & := 1 + \eta e^{r(T-t)} ( c(t)- d(t,u_t), \quad t \in [0,T], \label{a}\\
A(t,u.) & := \int_t^T a(s,u_s) e^{-\alpha (s-t)} \ud s, \quad t \in [0,T], \label{AB}\\
B(t,z, u.) & := e^{ \eta e^{r(T-t)} \Phi(z, u_t) - A(t,u.) \ell(z)} \quad t \in [0,T],\ z>0. \label{B}
\end{align}
\end{lemma}
\noindent The proof can be found in Appendix \ref{appendix:proofs}. Note that, $A(t,u.)$, and then $B(t,z, u.)$, depend on the path of $u$ over $[t,T]$.

 \noindent Next result furnishes a sufficient condition for the monotonicity property of the function $\varphi(t,\lambda)$.
\begin{proposition}
 Suppose Assumption \ref{A2_premium}  to be satisfied and for any $t \in (0,T)$ and $u \in \mathcal{U}$
\begin{equation}\label{strana}
 \int_0^{+ \infty} B(t,z, u.) F^\a(\ud z)   -  a(t,u_t) \geq 0, \quad \Pp-\mbox{a.s.},
 \end{equation}
 where $B(t,z, u.)$ and $a(t,u_t)$ are defined in \eqref{B} and  \eqref{a}, respectively. Then,  $\varphi(t, \lambda)$ is an increasing function of $\lambda \in (0, + \infty)$.
 \end{proposition}

\begin{proof}
Let us denote by $\widetilde \varphi(t, \lambda, u)$ the expectation in \eqref{direct}, that is
\begin{equation}
\begin{split}
\widetilde \varphi(t, \lambda, u)
  & := \mathbb{E}^\Q_{t,\lambda}\left[e^{\int_t^T \{ \int_0^{+ \infty} (e^{-A(s,u.) z} -1) \rho F^\b(\ud z) + \lambda_s \int_0^{+ \infty} B(s,z,u.) F^\a(\ud z) - a(s,u_s) [ \beta + (\lambda - \beta) e^{-\alpha (s-t)} ] \}\ud s } \right ].   \end{split}
\end{equation}
Denoting by $\{ \lambda_s^{t,\lambda};\ s \in [t, T]\}$ the solution of \eqref{intensity_eq} with the initial data $(t, \lambda) \in [0,T] \times (0, + \infty)$, we have the following.
\begin{equation} 
\begin{split}
  \widetilde \varphi(t, \lambda, u)
  & = \mathbb{E}^\Q \left[e^{\int_t^T \{ \int_0^{+ \infty} (e^{-A(s,u.) z} -1) \rho F^\b(\ud z) + \lambda^{t, \lambda}_s \int_0^{+ \infty} B(s,z,u.) F^\a(\ud z) - a(s,u_s) [ \beta + (\lambda - \beta) e^{-\alpha (s-t)} ] \}\ud s } \right ].
  \end{split}
  \end{equation}
From \eqref{triangle} in Appendix \ref{appendix:proofs}, we obtain
\begin{equation}\label{vartilde}
  \widetilde \varphi(t, \lambda, u)= \mathbb{E}^\Q \left[ H(t, T, u.) e^{\lambda \int_t^T  e^{-\alpha(s-t)}\{ \int_0^{+ \infty}  B(s,z,u.) F^\a(\ud z) - a(s,u_s) \} \ud s } \right], \end{equation}
where $H(t, T, u.)$ is given by
\begin{equation}
 \begin{split}
     H(t, T, u.)
     &=  e^{\int_t^T \{ \int_0^{+ \infty} (e^{-A(s,u.) z} -1) \rho F^\b(\ud z) - a(s, u_s) \beta(1- e^{-\alpha(s-t)} ) \} \ud s } e^{\int_t^T \beta(1- e^{-\alpha(s-t)})  \int_0^{+ \infty}  B(s,z,u.) F^\a(\ud z)}  \\
     & \quad \times e^{ \int_0^{+ \infty}  B(s,z,u.) F^\a(\ud z) \{ \int_t^s \int_0^{+ \infty} e^{-\alpha (s- v)} \ell(z) m^\a(\ud v , \ud z) + \int_t^s \int_0^{+ \infty} e^{-\alpha (s-v)} z m^\b(\ud v , \ud z)\} }
     \end{split}
     \end{equation} 
   and it does not depend on $\lambda$ because under $\Q$, $m^\a(\ud v , \ud z)$ and $m^\b(\ud v , \ud z)$ are Poisson random measures with deterministic compensators $F^\a(\ud z) \ud v$ and $\rho F^\b(\ud z) \ud v$, respectively.  
Hence, by \eqref{vartilde} we have that for  any $0<\lambda_1 < \lambda_2$ and  $u \in \mathcal{U}_t$, 
$\widetilde \varphi(t, \lambda_1, u) \leq \widetilde \varphi(t, \lambda_2, u).$
Finally, taking the infimum over $\mathcal{U}_t$ we obtain the thesis. i.e. 
$$ \varphi(t, \lambda_1)=\inf_{u \in \mathcal{U}_t} \widetilde \varphi(t, \lambda_1, u) \leq \inf_{u \in \mathcal{U}_t}\widetilde \varphi(t, \lambda_2,  u)= \varphi(t, \lambda_2).
$$
\end{proof}

\begin{remark}
Let us observe that condition \eqref{strana} involve both the insurance and reinsurance premium, the model's parameters, as well as the jump size $\ell(z)$ and the claim size distribution $F^\a$. 
From Propositions \ref{comparison} and \ref{comparison2} we have that, under the EVP for both, proportional reinsurance and limited-excess of loss, if condition \eqref{strana} is satisfied the insurer transfers more risk to reinsurer in the contagion model than in the Cox one.
This means that Insurance Company is not always conservative. A similar behaviour is observed in \citet{Cao_Landriault_Li} under the EVP and a Mean-Variance criterion.
 \end{remark}







\subsection*{Acknowledgements}

The authors are members of Gruppo Nazionale per l'Analisi Matematica, la Probabilità e le loro Applicazioni (GNAMPA) of Istituto Nazionale di Alta Matematica (INdAM). 


\subsection*{Funding}

The first author was partially supported by European Union-Next Generation EU - PRIN research project n. 2022BEMMLZ. 
The second author was partially supported by  the European Union-Next Generation EU - PRIN research project n. 2022FPLY97.

\bibliographystyle{plainnat}
\bibliography{biblioCC}

\begin{thebibliography}{20}
\providecommand{\natexlab}[1]{#1}
\providecommand{\url}[1]{\texttt{#1}}
\expandafter\ifx\csname urlstyle\endcsname\relax
  \providecommand{\doi}[1]{doi: #1}\else
  \providecommand{\doi}{doi: \begingroup \urlstyle{rm}\Url}\fi

\bibitem[Abdelhadi and Khelfallah(2022)]{abdel2022}
K.~Abdelhadi and N.~Khelfallah.
\newblock Locally lipschitz bsde with jumps and related kolmogorov equation.
\newblock \emph{Stochastics and Dynamics}, 22\penalty0 (05):\penalty0 2250021, 2022.

\bibitem[Albrecher and Asmussen(2006)]{Albrecher_Asmussen}
H.~Albrecher and S.~Asmussen.
\newblock Ruin probabilities and aggregate claims distributions for shot noise cox processes.
\newblock \emph{Scandinavian Actuarial Journal}, 2006\penalty0 (2):\penalty0 86--110, 2006.

\bibitem[Bandini et~al.(2024)Bandini, Confortola, and Di~Tella]{bandini2024compensator}
E.~Bandini, F.~Confortola, and P.~Di~Tella.
\newblock On the compensator of step processes in progressively enlarged filtrations and related control problems.
\newblock \emph{ALEA}, 21\penalty0 (1):\penalty0 95--120, 2024.

\bibitem[Bene{\v{s}}(1971)]{benevs1971existence}
V.~E. Bene{\v{s}}.
\newblock Existence of optimal stochastic control laws.
\newblock \emph{SIAM Journal on Control}, 9\penalty0 (3):\penalty0 446--472, 1971.

\bibitem[Bj\"{o}rk and Grandell(1988)]{Bjork_Grandell}
T.~Bj\"{o}rk and J.~Grandell.
\newblock Exponential inequalities for ruin probabilities in the cox case.
\newblock \emph{Scandinavian Actuarial Journal}, 1988\penalty0 (1-3):\penalty0 77--111, 1988.

\bibitem[Brachetta and Ceci(2019)]{Brachetta_Ceci_2019}
M.~Brachetta and C.~Ceci.
\newblock Optimal proportional reinsurance and investment for stochastic factor models.
\newblock \emph{Insurance: Mathematics and Economics}, 87:\penalty0 15--33, 2019.

\bibitem[Brachetta et~al.(2024)Brachetta, Callegaro, Ceci, and Sgarra]{brachetta_call_ceci_sgarra}
Matteo Brachetta, Giorgia Callegaro, Claudia Ceci, and Carlo Sgarra.
\newblock Optimal reinsurance via bsdes in a partially observable model with jump clusters.
\newblock \emph{Finance and Stochastics}, 28\penalty0 (2):\penalty0 453--495, 2024.

\bibitem[Cao et~al.(2020)Cao, Landriault, and Li]{Cao_Landriault_Li}
J.~Cao, D.~Landriault, and B.~Li.
\newblock Optimal reinsurance-investment strategy for a dynamic contagion claim model.
\newblock \emph{Insurance: Mathematics and Economics}, 93:\penalty0 206--215, 2020.

\bibitem[Cao et~al.(2023)Cao, Li, Young, and Zou]{cao2023stackelberg}
J.~Cao, D.~Li, V.~R. Young, and B.~Zou.
\newblock Stackelberg reinsurance chain under model ambiguity.
\newblock \emph{Scandinavian Actuarial Journal}, pages 1--32, 2023.

\bibitem[Confortola et~al.(2016)Confortola, Fuhrman, and Jacod]{cfj2016}
F.~Confortola, M.~Fuhrman, and J.~Jacod.
\newblock Backward stochastic differential equation driven by a marked point process: an elementary approach with an application to optimal control.
\newblock \emph{The Annals of Applied Probability}, 26\penalty0 (3):\penalty0 1743--1773, 2016.

\bibitem[Dassios and Zhao(2011)]{Dassios_Zhao_2011}
A.~Dassios and H.~Zhao.
\newblock A dynamic contagion process.
\newblock \emph{Advances in Applied Probabability}, 43:\penalty0 814--846, 2011.

\bibitem[Embrechts et~al.(1993)Embrechts, Schmidli, and Grandell]{Embrechts_Schmidli_Grandell}
P.~Embrechts, H.~Schmidli, and J.~Grandell.
\newblock Finite-time lundberg inequalities in the cox case.
\newblock \emph{Scandinavian Actuarial Journal}, 1993\penalty0 (1):\penalty0 17--41, 1993.

\bibitem[Hawkes(1971)]{Hawkes_1971}
A.~G. Hawkes.
\newblock Spectra of some self-exciting and mutually exciting point processes.
\newblock \emph{Biometrika}, 58\penalty0 (1):\penalty0 83--90, 1971.

\bibitem[Irgens and Paulsen(2004)]{Irgens_Paulsen}
C.~Irgens and J.~Paulsen.
\newblock Optimal control of risk exposure, reinsurance and investments for insurance portfolios.
\newblock \emph{Insurance: Mathematics and Economics}, 35\penalty0 (1):\penalty0 21--51, 2004.

\bibitem[Jeanblanc et~al.(2015)Jeanblanc, Mastrolia, Possama{\"\i}, and R{\'e}veillac]{jeanblanc2015}
M.~Jeanblanc, T.~Mastrolia, D.~Possama{\"\i}, and A.~R{\'e}veillac.
\newblock Utility maximization with random horizon: a bsde approach.
\newblock \emph{International Journal of Theoretical and Applied Finance}, 18\penalty0 (07):\penalty0 1550045, 2015.

\bibitem[Kazi-Tani et~al.(2015)Kazi-Tani, Possama{\"\i}, and Zhou]{kazi-tani2015}
M.~N. Kazi-Tani, D.~Possama{\"\i}, and C.~Zhou.
\newblock Quadratic bsdes with jumps: a fixed-point approach.
\newblock \emph{Electronic Journal of Probability}, 20\penalty0 (66):\penalty0 1--28, 2015.

\bibitem[Mania and Santacroce(2010)]{Mania2010}
M.~Mania and M.~Santacroce.
\newblock Exponential utility maximization under partial information.
\newblock \emph{Finance and Stochastics}, 14:\penalty0 419--448, 2010.

\bibitem[Papapantoleon et~al.(2018)Papapantoleon, Possamai, and Saplaouras]{Papa_Possa_Sapla2018}
A.~Papapantoleon, D.~Possamai, and A.~Saplaouras.
\newblock Existence and uniqueness results for bsde with jumps: the whole nine yards.
\newblock \emph{Electronic Journal of Probabability}, 23\penalty0 (121):\penalty0 1--68, 2018.

\bibitem[Quenez and Lim(2011)]{lim-quenez}
M.-C. Quenez and T.~Lim.
\newblock Exponential utility maximization in an incomplete market with defaults.
\newblock \emph{Electronic Journal of Probability}, 16\penalty0 (53):\penalty0 1434--1464, 2011.

\bibitem[Wu et~al.(2024)Wu, Shen, Zhang, and Ding]{wu2024optimal}
F.~Wu, Y.~Shen, X.~Zhang, and K.~Ding.
\newblock Optimal claim-dependent proportional reinsurance under a self-exciting claim model.
\newblock \emph{Journal of Optimization Theory and Applications}, pages 1--27, 2024.

\end{thebibliography}

\appendices

\section{Technical proofs and auxiliary results} \label{appendix:proofs}

\noindent First, we provide the proof of the Verification Theorem.

\begin{proof}[Proof of Theorem \ref{Verifica}]
First, observe that if $\widetilde \varphi (t,\lambda) \in C^1((0,T) \times (0,+\infty)) \cap C([0,T] \times(0,+\infty))$ solves \eqref{eq:HJB2} - \eqref{eq:final_cond} then $\widetilde v(t,x,\lambda) = e^{-\eta x e^{r(T-t)}} \widetilde\varphi(t,\lambda)$ solves the HJB-equation  \eqref{HJB1}.
From It\^o's formula  we get that, for any $0\leq t\leq T$  and $u \in \mathcal{U}$, it holds
\begin{equation}\label{eq:function_f}
    \widetilde v(T,X^u_T,\lambda_T)  = \widetilde v(t, X^u_t, \lambda_t) + \int_t^T{\Ll}^{X,\lambda,u} \widetilde v(s,X^u_s, \lambda_s) \ud s + M_T - M_t,
\end{equation}
where  
\begin{align}
M_t = & \int_0^t \int_{\R^+} 
\left[\widetilde v(s, X^u_{s^-}-z u_s,\lambda_{s^-} + \ell(z))-\widetilde v(s,X^u_{s^-},\lambda_{s^-})\right] (m^\a(\ud s, \ud z) - \lambda_{s^-} F_Z(\ud z) \ud s)\\
& \quad + \int_0^t\int_0^{+ \infty} \left[ \widetilde v(s,X^u_{s^-}, \lambda_{s^-} + z) - \widetilde v(s,X^u_{s^-},\lambda_{s^-}) \right] (m^\b(\ud s, \ud z) -\rho F^\b (\ud z) \ud s),
\end{align}
for $t \in [0,T]$.
We introduce the non-decreasing sequence of stopping times defined as 
\begin{equation}
\tau_n = \inf\left\{t\in [0,T] : |X^u_t| > n \vee \lambda_t> n \vee \lambda_t < \frac{1}{n}\right\}. 
\end{equation}
Since both $X^u$ and $\Lambda$ do not explode and $\Lambda$ is strictly positive, we have $n\to \infty$, $\tau_n\to T$.
By assumption, $\widetilde v(t,x,\lambda)$ is continuous and hence bounded in compact sets. Therefore, the stopped process $\{M_{t\wedge \tau_n},\ t \in [0,T]\}$ is an $(\bF, \Pp)$-martingale.
Indeed, for every $n \in \bN$, denoting $R_n = [0,T]\times [-n,n] \times [\frac{1}{n},n]$, the following conditions are satisfied:
\begin{align}
& \esp{\int_0^{\tau_n} \int_{\R^+} \left|\widetilde v(s,X^{u}_{s^-} -z u_s, \lambda_{s^-} + \ell(z)) - \widetilde v(s,X^u_{s^-}, \lambda_{s^-}) \right | \lambda_s F^\a_Z(\ud z)  \ud s} \\
& \qquad \qquad \leq \sup_{(t,x,\lambda)\in R_n} 2|\widetilde v(t,x, \lambda) | n T < + \infty\\
&\mathbb{E}\left[\int_0^{\tau_n} \int_{\R^+} \left|\widetilde v(s,X^u_{s^-}, \lambda_{s^-} + z) - \widetilde v(s,X^u_{s^-}, \lambda_{s^-}) \right | \rho F^\b(\ud z) \ud s\right]\\
& \qquad \qquad \leq \sup_{(t,x,\lambda)\in R_n} 2 |\widetilde v(t,x, \lambda)| \rho T < +\infty,
\end{align}
which guarantee that $\left\{M_{\tau_n \wedge t},\ t \in [0,T]\right\}$ is an $(\bF, \Pp)$-martingale.  

From \eqref{HJB1} it holds that for any $u \in \mathcal{U}$, and $s \in [t,T]$,  ${\Ll}^{X,\lambda,u} \widetilde v(s,X^u_s, \lambda_s) \geq 0$, $\Pp$-a.s.. Therefore,  taking the conditional expectation of both sides of \eqref{eq:function_f} with $T$ replaced by $T\wedge \tau_n$ and $t$ by $t\wedge \tau_n$, we get that
\begin{equation}\label{EXP}
    \mathbb{E}_{t,x, \lambda}\bigl[ \widetilde v(T\wedge \tau_n,X^u_{T\wedge \tau_n},\lambda_{T\wedge \tau_n}) \bigr] \geq \mathbb{E}_{t,x, \lambda}\bigl[ \widetilde v(t\wedge \tau_n, X^u_{t\wedge \tau_n},\lambda_{t\wedge \tau_n}) \bigr],
    \end{equation}
where $(t, x, \lambda) \in [0,T]\times \mathbb R \times (0, + \infty)$.
    Letting $n \rightarrow + \infty$, and using the fact that the process $X^{u}$ and $\Lambda$ do not have any deterministic jump time it holds that 
    $$\widetilde v(T\wedge \tau_n,X^u_{T\wedge \tau_n},\lambda_{T\wedge \tau_n}) \rightarrow e^{-\eta X^{u}_T}, \ \Pp-\mbox{a.s.},$$ and 
    $$\widetilde v(t\wedge \tau_n, X^u_{t\wedge \tau_n},\lambda_{t\wedge \tau_n})\rightarrow \widetilde v(t, X^u_{t},\lambda_{t}), \ \Pp-\mbox{a.s.},$$
    for 
    any $u \in \mathcal{U}$.   
    From assumption (ii) we can now apply the limit under expectation in  \eqref{EXP}, thus 
    \begin{equation}\label{ver}\mathbb{E}_{t,x, \lambda}\bigl[e^{-\eta X^{u}_T}\bigr] \geq  \widetilde v(t,x,\lambda),\end{equation}
    which implies $v(t,x, \lambda) \geq \widetilde  v(t, x, \lambda)$.  By the continuity of $\Phi(z,u)$ and $q(\lambda,u)$ with respect to $u \in U$ and since $U$ is compact, there exists a measurable function $u^*(t,\lambda)$ which realizes the infimum of \eqref{Psiu}. The control $\{u^*(t,\lambda_{t^-}), \, t \in [0,T]\} \in \mathcal{U}$ is admissible from Proposition \ref{ADM} (b). Finally, by computations similar to those above, we can prove that equality holds in \eqref{ver} when taking the control $\{u^*(t,\lambda_{t^-}), \, t \in [0,T]\}$.
    Consequently, 
   $$v(t,x,\lambda) = \mathbb{E}_{t,x,\lambda}\bigl[e^{-\eta X^{u^*}_T}\bigr] = \widetilde v(t, x, \lambda),$$
   which concludes the proof. 
   \end{proof}

\begin{proof}[Proof of Theorem \ref{T2}]
In order to apply \citet[Theorem 3.5]{Papa_Possa_Sapla2018} we start by rewriting the BSDE in terms of the integer-valued random measure  $m(\ud t, \ud z_1, \ud z_2)$ defined in \eqref{m},
\begin{equation} \label{bsde1}
\begin{split}
Y_t  = \xi - \int_t^T \int_0^{+ \infty} \int_0^{+ \infty}\Theta^{Y}_{s^-} (z_1,z_2) \widetilde m(\ud s, \ud z_1, \ud z_2) - \int_t^T \esssup_{u \in \mathcal U} \widetilde F( s, Y_{s^-} , \Theta^{Y}_s(\cdot, \cdot),u_s)\,\ud s 
 \end{split}
\end{equation}
 where
 $\widetilde F( t, Y_{t^-} , \Theta^{Y}_t(\cdot, \cdot),u_t) = \widetilde{f} ( t, Y_{t^-} , \Theta^{Y}_t(\cdot, 0 ) , u_t )$
and  $\widetilde{f} ( t, Y_{t^-} , \Theta^{Y}_t(\cdot, 0 ) , u_t )$ is given in \eqref{ftilde}.

\noindent Note that BSDE \eqref{bsde1} is a special case of that considered in  \citet{Papa_Possa_Sapla2018}.  Precisely, it is only driven by an integer-valued random measure with $\bF$-predictable compensator absolutely continuous with respect to the Lebesgue measure and the orthogonal martingale term is not present thanks to the $(\bF,\Pp)$-martingale representation property in  Proposition \ref{mg representation}.

\noindent We now verify that assumptions $\mathbf{(F1)}$-$\mathbf{(F5)}$ in  \citet[Theorem 3.5]{Papa_Possa_Sapla2018} are satisfied.  

\noindent $\mathbf{(F1)}$  The process $\widetilde Z= \{\widetilde Z_t = (\widetilde C^\a_t, \widetilde C^\b_t), \ t\in [0,T]\}$ where
$$ \widetilde C^{(i)}_t = \int_0^t \int_0^{+ \infty} z_i \widetilde m^{(i)}(\ud s, \ud z_i), \quad i =1,2$$ 
is a two-dimensional pure jump $(\bF, \Pp)$-martingale (see Proposition \ref{mom}) such that  
$\sup_{t \in [0,T]}\mathbb{E}\left[\|Z\| ^2\right] < +\infty$. Hence \citet[Assumption 2.10]{Papa_Possa_Sapla2018} is satisfied.

\noindent Indeed, for any $t \in [0,T]$
\begin{equation} \begin{split}
\mathbb{E}\left[ (\widetilde C^\a_t)^2  \right] =&
 \mathbb{E}\left[ \int_0^t \int_0^{+\infty} z^2 \lambda_{s^-} F^\a(\ud z) \ud s \right] = \mathbb{E}\left[(Z^\a)^2\right] \mathbb{E} \left[ \int_0^T \lambda_s \ud s \right]  \\
\mathbb{E}\left[ (\widetilde C^\b_t)^2  \right] =&
 \mathbb{E}\left[ \int_0^t \int_0^{+\infty} z^2   \rho F^\b(\ud z) \ud s \right] = \mathbb{E}\left[(Z^\b)^2\right] \rho T,
\end{split}
\end{equation}
thus
\begin{equation} \begin{split}
\sup_{t \in [0,T]}\mathbb{E}\left[ \| \widetilde Z_t \|^2\  \right] &= 
 \sup_{t \in [0,T]}\left( \mathbb{E}\left[ (\widetilde C^\a_t)^2  \right]  + \mathbb{E}\left[ (\widetilde C^\b_t)^2  \right] \right)\\
& \le \mathbb{E}\left[(Z^\a)^2\right] \mathbb{E} \left[ \int_0^T \lambda_s \ud s \right]  + \mathbb{E}\left[(Z^\b)^2\right] \rho T,
\end{split}
\end{equation}
which is finite thanks to Proposition \ref{mom}.

\noindent The disintegration property is fulfilled because 
$\nu(\omega, \ud t, \ud z_1, \ud z_2) = K^{\omega} (\ud z_1, \ud z_2) \ud t,
$
where the transition kernel
$K^{\omega}$ on $(\Omega \times [0,T], {\mathcal P})$ (here ${\mathcal P}$ denotes the $\bF$-predictable sigma-algebra on $\Omega \times [0,T]$) is given by 
$$
K^{\omega} (\ud z_1, \ud z_2)= \lambda_{t^-}(\omega) F^\a(\ud z_1) \delta_{0}(\ud z_2) + \rho F^\b(\ud z_2) \delta_{0}(\ud z_1).
$$

\noindent $\mathbf{(F2)}$ The terminal condition of the BSDE $\xi = e^{- \eta X^N_T}$ has finite moments for any order. Indeed, since $\Phi(z,u_N)=z$ and $q_t^{u_N}\equiv 0$ for every $t \in [0,T]$, we have
\begin{align}
 \esp{ \xi^p} & = \esp{ e^{-p\eta X^N_T}}\\
 &=\esp{e^{-p\eta R_0e^{rT}}e^{-p\eta\int_0^T e^{r(T-s)} c_s \,\ud s}
e^{p\eta\int_0^T\int_0^{+\infty} e^{r(T-s)} z  \,m^{(1)}(\ud s,\ud z)}}\\
& \leq \esp{
e^{p\eta e^{r T}\int_0^T\int_0^{+\infty}  z  \,m^{(1)}(\ud s,\ud z)}}= \esp{e^{p\eta e^{r T}C_T}} < +\infty,\label{momenti}
\end{align}
for every $p >0$, in view of Proposition \ref{ADM} (a). See also $\mathbf{(F4)}$ below for additional details.

\noindent $\mathbf{(F3)}$ The generator $F(t, \omega, y, \theta  ( \cdot, \cdot )):= \esssup_{u \in U} \widetilde F ( t, \omega, y, \theta  ( \cdot, \cdot )), u )$ defined on the space $$\mathbb{M} = \{ (t,\omega,y,\theta(\cdot,\cdot): (t,\omega,y)\in[0,T]\times\Omega\times (0,+\infty) \text{ and } \theta(\cdot,\cdot): [0,+\infty)^2\to\mathbb{R},\ \textrm{measurable} \}$$
satisfies a stochastic Lipschitz condition, i.e., there exist two positive $\mathbb{F}$-predictable processes $\gamma, \bar{\gamma}$ such that 
\begin{equation}\label{eqn;stochlip}
\big| F(t, \omega, y, \theta  ( \cdot, \cdot ) ) - F(t, \omega, y', \theta' (\cdot, \cdot ) ) \big| ^2
	\leq \gamma_t (\omega ) |y - y'|^2
	+\bar{\gamma}_t (\omega ) \left( ||| \theta ( \cdot, \cdot ) - \theta' ( \cdot, \cdot ) |||_t (\omega ) \right)^2 ,
\end{equation}
where:
\begin{equation}
\begin{split}
\left( ||| \theta ( \cdot, \cdot ) |||_t (\omega ) \right)^2 := & \int_0^{+\infty} \int_0^{+\infty} \theta^2(z_1,z_2) K_t^{\omega}(\ud z_1, \ud z_2) \\  = & \int_0^{+\infty} \theta^2(z_1,0) \lambda_{t^-}(\omega) F^\a(\ud z_1) + 
\int_0^{+\infty} \theta^2(0,z_2) \rho F^\b(\ud z_2).
\end{split}
\end{equation}
Exploiting the definition of $F$, we first need to deal with the essential supremum:
\[
\bigl| F (t, \omega, y, \theta  ( \cdot, \cdot ) ) - F(t, \omega, y', \theta' (\cdot , \cdot)) \bigr| ^2
	\leq \left( \esssup_{u \in U} \bigl| \widetilde F ( t, \omega, y, \theta  ( \cdot, \cdot ), u ) -  \widetilde F ( t, \omega, y', \theta'  ( \cdot, \cdot ), u ) \bigr| \right)^2,
\]
and we preliminarily work on the absolute value difference involving $\widetilde{F}$:
\[
\begin{split}
&\bigl| \widetilde{F} ( t, \omega, y, \theta  ( \cdot, \cdot ), u ) -  \widetilde{F} ( t, \omega, y', \theta'  ( \cdot, \cdot ), u ) \bigr| = \bigg| (y- y') \eta e^{r (T-t)} q^u_t(\omega) \\
& \quad + \int_{0}^{+\infty} \left( y- y' + \theta(z_1, 0) - \theta'(z_1, 0)\right) \left( e^{-\eta e^{r(T-t)}(z-\Phi(z,u))} - 1 \right) \lambda_{t^-} F^\a (\ud z_1) \bigg| \\
	& \quad \le \bigl| y- y' \bigr| \eta e^{r (T-t)} q^{u_M}_t (\omega)
	+  \left| y- y' \right| \lambda_{t^-}
	+  \int_{0}^{+\infty} \left| \theta(z_1, 0) - \theta'(z_1,0) \right| \lambda_{t^-} F^\a (\ud z_1)\\
\end{split}
\]
since $| e^{-\eta e^{R(T-t)}(z-\Phi(z,u))} - 1 |\leq 1$ and $q_t^u \leq q^{u_M}_t$ for any $u \in U$. Now, since the inequality above does not depend on $u$ we also have that the $\esssup_{u \in U}$ satisfies it and we can take its square  (we use here the trivial relation $ (a+ b+c)^2 \leq 3 (a^2 + b^2 + c^2)$), finding:
\[
\begin{split}
& \left( \esssup_{u \in U} \bigl| \widetilde{F} ( t, \omega, y, \theta  ( \cdot, \cdot ), u ) -  \widetilde{F} ( t, \omega, y', \theta'  ( \cdot, \cdot ), u ) \bigr| \right)^2 \le 3 \bigl| y- y' \bigr|^2 \eta^2 e^{2r (T-t)} (q^{u_M}_t (\omega))^2
	\\
	& \quad + 3  \left| y- y' \right|^2   \lambda^2_{t^-} + 3 \left( \int_{0}^{+\infty} \left| \theta(z_1, 0) - \theta'(z_1, 0) \right| \lambda_{t^-} F^\a (\ud z_1) \right)^2. \\
\end{split}
\]
We now  apply Jensen inequality to the last term and we get
\[
\begin{split}
& \left( \esssup_{u \in U} \bigl| \widetilde{F} ( t, \omega, y, \theta  ( \cdot, \cdot ), u ) -  \widetilde{F} ( t, \omega, y', \theta'  ( \cdot, \cdot ), u ) \bigr| \right)^2 \\
	&\le 3\bigl| y- y' \bigr|^2 \eta^2 e^{2r (T-t)} (q^{u_M}_t (\omega))^2
	+ 3 \left| y- y' \right|^2 \lambda^2_{t^-}(\omega)+ 3 \int_{0}^{+\infty} \left| \theta(z_1, 0) - \theta'(z_1, 0) \right|^2 \lambda^2_{t^-} (\omega)   F^\a(\ud z_1)\\
	&= 3\bigl| y- y' \bigr|^2 \left( \eta^2 e^{2r (T-t)} \left(q^{u_M}_t (\omega) \right)^2 +  \lambda^2_{t^-}(\omega) \right)
	+ 3 \lambda_{t^-} (\omega)\left( ||| \theta ( \cdot, \cdot ) - \theta' ( \cdot, \cdot ) |||_t (\omega ) \right)^2 .
\end{split}
\]
So, the target, being Equation \eqref{eqn;stochlip}, is reached and we have the following values for the stochastic Lipschitz coefficients $\gamma_t$ and $\bar \gamma_t$:
\begin{align*}
\gamma_t &= 3 \eta^2 e^{2r (T-t)} ( q^{u_M}_t )^2 + 3 \lambda_{t^-}^2 , \quad \bar{\gamma}_t = 3 \lambda_{t^-},
\end{align*}
which, as expected, are independent of the control $u$.
\item[$\mathbf{(F4)}$] Since by definition $ \alpha _{\cdot} ^2 = \max \{ \sqrt{\gamma_\cdot}, \bar \gamma_{\cdot} \}$, here we find:
$$
\alpha_s^2 =  \max \left\{ \sqrt{3 \eta^2 e^{2r (T-s)} ( q^{u_M}_s )^2 + 3 \lambda_{s^-}^2 }, 3\lambda_{s^-}\right\}
$$
and also $A_t = \int_{0}^{t} \alpha _{s}^2 \,\ud s $, so that we can easily verify that the inequality $\Delta A_t \leq \Phi, \Pp-$a.s. holds true for any $\Phi >0$ since $A$ has no jumps.
Notice that $\mathbf{(F2)}$ requires that the terminal condition $\xi = e^{- \eta X^T_N}$ belongs to the set of  $\mathcal F_T-$measurable random variables such that $\mathbb E \left[ e^{\widehat \beta A_T} \xi^2 \right] < \infty$, for some $\widehat \beta > 0$.
This is true for any $\widehat \beta >0$, since $\alpha_s^2 \le \sqrt 3 \eta e^{r (T-s)}  q^{u_M}_s  + 3 \lambda_{s^-} $ and so
\begin{align}
\mathbb E \left[ e^{\widehat \beta A_T} \xi^2 \right] &\le \frac{1}{2}\mathbb E \left[ e^{ 2 \widehat \beta A_T}\right] + \frac{1}{2}\mathbb E \left[\xi^4 \right] \\
&= \frac{1}{2} \mathbb E \left[ e^{ 2\widehat \beta \sqrt 3 \eta  \int_{0}^{T}  e^{r (T-s)}  q^{u_M}_s \ud s} e^{ 6 \widehat \beta \int_0^T  \lambda_{s}\ud s }\right] +   \frac{1}{2}\mathbb E \left[\xi^4 \right]\\
&\le \frac{1}{4}\mathbb E \left[ e^{ 4\widehat \beta \sqrt 3 \eta  \int_{0}^{T}  e^{r (T-s)}  q^{u_M}_s \ud s} \right] +  \frac{1}{4} \mathbb E \left[ e^{ 12 \widehat \beta \int_0^T  \lambda_{s}\ud s }\right] +  \frac{1}{2}\mathbb E \left[\xi^4 \right].  
\end{align}
Thus by 
Assumption \ref{ass_app_premium} $(ii)$, equation \eqref{momenti} and Proposition \ref{ADM} (a), we get that for any $\widehat \beta >0$, $\mathbb E \left[ e^{\widehat \beta A_T} \xi^2 \right] < +\infty$.
\item[$\mathbf{(F5)}$] Finally, by using the same $\widehat\beta >0$ and $A$ introduced to prove $\mathbf{(F4)}$, we find:
\begin{equation}
\mathbb{E} \left[  \int_0^T e^{\widehat\beta A_t} \frac{|F (t, 0, 0, 0) |^2}{\alpha_t^2}  dt \right] < \infty,
\end{equation}
since here $F (t,0,0) = - \esssup_{u \in \mathcal{U}} \widetilde{F} (t, 0, 0, u_t) = 0$.

It now remains to prove that the quantity
$$
M^{\Phi}(\widehat \beta) = \frac{9}{\widehat\beta} + \frac{\Phi^2 (2 + 9 \widehat \beta)}{\sqrt{\widehat \beta^2 \Phi^2 +4} -2} \exp{\left( \frac{\widehat \beta \Phi + 2 -\sqrt{\widehat \beta^2 \Phi^2 +4}}{2} \right)}
$$
with $\Phi>0$ introduced in $\mathbf{(F4)}$ and $\widehat \beta >0$, satisfies $M^{\Phi}(\widehat \beta) < \frac12$.
Thanks to \citet[Lemma 3.4]{Papa_Possa_Sapla2018}, for $\widehat \beta$ sufficiently large, we know that since $\lim_{\widehat \beta \rightarrow \infty} M^\Phi(\widehat \beta) = 9 e \Phi$ then it suffices to take $\Phi < \frac{1}{18 e}$.
According to \citet[Theorem 3.5]{Papa_Possa_Sapla2018} there exists a unique solution  $(Y,\Theta^{Y})$  to BSDE \eqref{bsde1} such that for any  $\widehat\beta>0$
\[
\mathbb{E}\left[ \int_0^T e^{\widehat\beta A_t} \alpha_t^2 \abs{Y_t}^2 \,dt\right] <+\infty, \quad \mathbb{E}\left[ \int_0^T \int_0^t \int_0^{+\infty} e^{\widehat\beta A_t}  \abs{\Theta_s^Y(z_1,z_2)}^2 \nu ( \ud s, \ud z_1, \ud z_2) \right] <+\infty.
\]
We notice that $\alpha^2_t \ge 3 \lambda_{t^-} \ge 3 \min\{ \lambda_0, \beta \}$ and this implies $ \mathbb{E}\left[ \int_0^T e^{\widehat\beta A_t} \abs{Y_t}^2 \,\ud t\right] <+\infty$ and therefore $Y \in \mathcal{L}^2$. 
By recalling the structure of $\nu ( \ud s, \ud z_1, \ud z_2)$ we finally obtain that $\Theta^{Y}(\cdot, 0)\in\widehat{\mathcal{L}}^\a$
and $\Theta^{Y}(0, \cdot)\in\widehat{\mathcal{L}}^\b$ and this concludes the proof.



 \end{proof}

\begin{proof}[Proof of Lemma \ref{PROP1}]
We perform a computation of the expectation in \eqref{varphi}  by applying the change of probability measure from $\Pp$ to $\Q$ defined by \eqref{eqn:L}. Under Assumption \ref{A2_premium} we have
\begin{equation}\label{CN1}
\begin{split}
&\mathbb E_{t,\lambda} \left[  e^{-\eta\int_t^T e^{r(T-s)} \left( c(s,\lambda_s)-q(s,\lambda_s,u_s) \right) \,\ud s 
+\eta\int_t^T\int_0^{+\infty} e^{r(T-s)} \Phi(z, u_s)  \,m^{(1)}(\ud s,\ud z)}\right]  = \\
& \mathbb E_{t,\lambda}^{\Q} \left[  e^{ -\int_t^T (\lambda_s - 1) ds  + \int_t^T  \ln (\lambda_{s^-}) d N_s^{(1)} }  e^{-\eta\int_t^T e^{r(T-s)} \lambda_s \left( c(s) -d(s,u_s) \right) \,\ud s 
+\eta\int_t^T\int_0^{+\infty} e^{r(T-s)} \Phi(z, u_s)  \,m^{(1)}(\ud s,\ud z)}\right] 
\\
& =e^{T-t} \, \mathbb E_{t,\lambda}^{\Q} \left[  e^{ -\int_t^T \lambda_s  a(s,u_s) \ud s 
+ \int_t^T   \int_0^{+\infty} [ \ln (\lambda_{s^-}) + \eta e^{r(T-s)} \Phi(z, u_s) ] \,m^{(1)}(\ud s,\ud z)}\right] .
\end{split}
\end{equation}
Let us denote by $\{ \lambda_s^{t,\lambda};\ s \in [t, T]\}$ the solution of Equation \eqref{intensity_eq} with initial data $(t, \lambda) \in [0,T] \times (0, + \infty)$. By standard computations we have for any $s \in [t,T]$
\begin{equation}\label{triangle}
\lambda^{t,\lambda}_s = \beta + (\lambda - \beta)  e^{-\alpha (s-t)} + \int_t^s \int_0^{+ \infty} e^{-\alpha (s- v)} \ell(z) m^\a(\ud v , \ud z) + \int_t^s \int_0^{+ \infty} e^{-\alpha (s-v)} z m^\b(\ud v , \ud z).
\end{equation}
As a consequence we get that  
\begin{equation}  
\begin{split}\label{aL1}
&\int_t^T a(s,u_s) \lambda^{t,\lambda}_s \ud s  = \int_t^T a(s,u_s) ( \beta + (\lambda - \beta) e^{-\alpha (s-t)} )\ud s  \\
& +\int_t^T \int_t^s \int_0^{+ \infty} a(s,u_s) e^{-\alpha (s-v)}  \ell(z) m^\a(\ud v , \ud z) \ud s  
+\int_t^T \int_t^s \int_0^{+ \infty} a(s,u_s) e^{-\alpha (s-v)} z m^\b(\ud v , \ud z) \ud s.  
\end{split}
\end{equation}
By applying Fubini's Theorem, we can write
\begin{align*}
& \int_t^T \int_t^s \int_0^{+ \infty} a(s,u_s) e^{-\alpha (s-v)}  \ell(z) m^\a(\ud v , \ud z) \ud s  +\int_t^T \int_t^s \int_0^{+ \infty} a(s,u_s) e^{-\alpha (s-v)} z m^\b(\ud v , \ud z) \ud s \\
&=\int_t^T  \int_0^{+ \infty} \left \{\int_v^T a(s,u_s) e^{-\alpha (s-v)} \ud s \right \} \ell(z) m^\a(\ud v , \ud z)  \\
& \qquad + 
\int_t^T \int_0^{+ \infty} \left\{\int_v^T a(s,u_s) e^{-\alpha (s-v)} \ud s \right \} z m^\b(\ud v , \ud z).
\end{align*} 
Hence, \eqref{aL1} reads as
\begin{equation}\label{int1}
\begin{split}
\int_t^T a(s,u_s) \lambda^{t,\lambda} _s \ud s & =  \int_t ^T a(s,u_s) ( \beta + (\lambda - \beta) e^{-\alpha (s-t)} )\ud s  \\
& +\int_t^T  \int_0^{+ \infty} A(v,u.)  \ell(z) m^\a(\ud v , \ud z) + \int_t^T  \int_0^{+ \infty} A(v,u.) z m^\b(\ud v , \ud z).
\end{split}
\end{equation}
Plugging \eqref{int1} into the last expectation in Equation \eqref{CN1} we obtain that
 \begin{align}
\varphi(t, \lambda)& = \inf_{u\in\mathcal{U}_t} e^{T-t} \, \mathbb E_{t, \lambda}^{\Q} \bigg[  e^{- \int_t^T\! a(s,u_s) ( \beta + (\lambda - \beta) e^{-\alpha (s-t)} )\ud s } \times \\ 
& e^{ -  \int_t^T \! \int_0^{+ \infty}\! A(s,u.) z m^\b(\ud s , \ud z) }  \,  e^{\int_t^T\!\int_0^{+\infty}\! \left( \ln (\lambda_{s^-})  + \eta e^{r(T-s)} \Phi(z, u_s) - A(s,u.) \ell(z)\right)  \,m^{(1)}(\ud s,\ud z)} \bigg].
\end{align} 
Finally, from Lemma \ref{lemma_eq} below, we get the thesis.
\end{proof}
\noindent Before proving Lemma \ref{lemma_eq}, we  recall the following results proved in \citet{brachetta_call_ceci_sgarra}.
 \begin{lemma}\label{lemma:exp_predictable2}
Let $(\Omega,\F, \Pp;\bF)$ be a filtered probability space and assume that the filtration $\bF=\{\F_t, \ t \in [0,T]\}$ satisfies the usual hypotheses. Let $N(\ud t,\ud z)$ be a Poisson random measure on $[0,T]\times[0,+\infty)$ with $\bF$-intensity kernel $\lambda F(\ud z)\,\ud t$. Then, for any $\bF$-predictable and $[0,+\infty)$-indexed process $\{H(t,z),\ t \in [0;T],\ z \in [0,+\infty) \}$ we have that
$$
\mathbb E \left[ e^{\int_0^T\int_0^{+\infty} H(t,z)\,N(\ud t,\ud z) }\right]
	= \mathbb E \left[ e^{\int_0^T\int_0^{+\infty} ( e^{H(t,z)} -1) \lambda F(\ud z)\,\ud t }\right] ,
$$
provided that the last expectation is finite.

\end{lemma}
\begin{lemma}\label{lemma_eq}
For any $u \in \mathcal{U}_t$, $(t, \lambda) \in [0,T] \times (0,+ \infty)$ the following equality holds
 \begin{equation*}
 \begin{split}
& \mathbb E_{t, \lambda}^{\Q}\! \left[ e^{ -  \int_t^T \! \int_0^{+ \infty}\! A(s,u.) z m^\b(\ud s , \ud z) } 
e^{\int_t^T\!\int_0^{+\infty}\! \left( \ln (\lambda_{s^-})  + B(s,z, u.) \right ) \,m^{(1)}(\ud s,\ud z)}  e^{- \int_t^T\! a(s,u_s) ( \beta + (\lambda - \beta) e^{-\alpha (s-t)} )\ud s }\right ] =\\
& e^{-(T-t)} \mathbb E_{t, \lambda}^{\Q} \left[ e^{ -  \int_t^T  \int_0^{+ \infty} (e^{-A(s,u.) z} -1) \rho F^\b(\ud z)\ud s } 
e^{\int_t^T\int_0^{+\infty} \lambda_{s}  e^{ B(s,z, u.) } \,F^\a (\ud z) \ud s}  e^{- \int_t^T a(s,u_s) ( \beta + (\lambda - \beta) e^{-\alpha (s-t)} )\ud s } \right ].
\end{split}
 \end{equation*}
\end{lemma}
\begin{proof}
    Since the random measures $m^\a(\ud t , \ud z)$ and $m^\b(\ud t , \ud z)$ are independent Poisson random measures with intensity kernel $ F^\a(\ud z)\,\ud t$ and $F^\b(\ud z) \rho \,\ud t$, respectively, under $\Q$, by interponing the conditioning on $\G^u \vee \G^\lambda$, with $\G^u=\sigma\{u_s,\ t \le s \le T\}$ and $\G^\lambda=\sigma\{\lambda_s,\ t \le s \le T\}$,  we get
    \begin{align*}
& \mathbb E_{t, \lambda}^{\Q}\! \left[ e^{ -  \int_t^T \! \int_0^{+ \infty}\! A(s,u.) z m^\b(\ud s , \ud z) } 
e^{\int_t^T\!\int_0^{+\infty}\! \left( \ln (\lambda_{s^-})  + B(s,z,u.) \right)  \,m^{(1)}(\ud s,\ud z)}  e^{- \int_t^T\! a(s,u_s) ( \beta + (\lambda - \beta) e^{-\alpha (s-t)} )\ud s }\right ] \\
&= E_{t, \lambda}^{\Q}\!\bigg[ e^{- \int_t^T\! a(s,u_s) ( \beta + (\lambda - \beta) e^{-\alpha (s-t)} )\ud s}  E_{t, \lambda}^{\Q}\!\left [e^{ -  \int_t^T \! \int_0^{+ \infty}\! A(s,u.) z m^\b(\ud s , \ud z) }|\G^u \vee \G^\lambda \right ] \\
& \qquad \times E_{t, \lambda}^{\Q}\!\left [e^{\int_t^T\!\int_0^{+\infty}\! \left( \ln (\lambda_{s^-})  + B(s,z,u.)\right)  \,m^{(1)}(\ud s,\ud z)}|\G^u \vee \G^\lambda \right ]  \bigg]\\
& = \mathbb E_{t, \lambda}^{\Q}\!\bigg[ e^{- \int_t^T\! a(s,u_s) ( \beta + (\lambda - \beta) e^{-\alpha (s-t)} )\ud s}  \mathbb E_{t, \lambda}^{\Q} \left [e^{ \int_t^T  \int_0^{+ \infty} (e^{-A(s,u.) z} -1) \rho F^\b(\ud z)\ud s }  |\G^u \vee \G^\lambda\right ]\\
& \qquad \times \mathbb E^{\Q} \left [ e^{\int_t^T\int_0^{+\infty} (\lambda_{s}  e^{ B(s,z,u.)}-1) \,F^\a (\ud z) \ud s} |\G^u \vee \G^\lambda \right]\bigg],
\end{align*}
where the last equality holds in view of Lemma \ref{lemma:exp_predictable2}. Finally, the thesis follows. 
\end{proof} 

\end{document}